\documentclass[11pt]{amsart}  

\usepackage[utf8]{inputenc}

\usepackage[left=1in,right=1in,top=1in,bottom=1.17in]{geometry}

\usepackage{graphicx,amssymb,amsmath,amsthm,wrapfig}
\usepackage{caption}
\usepackage{subcaption}
\usepackage{indentfirst}
\usepackage{cancel}
\usepackage{lipsum}
\usepackage{epstopdf}
\usepackage{epsfig}
\usepackage{comment}
\usepackage{enumerate,color}
\usepackage{empheq}
\usepackage{color}
\usepackage{hyperref}

\theoremstyle{plain}
\newtheorem{theorem}{Theorem}[section]
\newtheorem{corollary}[theorem]{Corollary}
\newtheorem{lemma}[theorem]{Lemma}
\newtheorem{definition}[theorem]{Definition}
\newtheorem{proposition}[theorem]{Proposition}

\newtheorem*{propositionprime}{Proposition \ref{fourth}$'$}
\numberwithin{equation}{section}

\theoremstyle{definition}
\newtheorem{remark}[theorem]{Remark}

\newcommand{\RNum}[1]{\uppercase\expandafter{\romannumeral #1\relax}}

\title[$L^2$-contraction and stability of Cahn--Hilliard fronts]{$L^2$-contraction and asymptotic stability \\ of Cahn--Hilliard fronts}

\author[B. Kwon and W. Shim]{Bongsuk Kwon and Wanyong Shim}

\address{(Bongsuk Kwon) Department of Mathematical Sciences, Ulsan National Institute of Science and Technology, Ulsan, 44919, Korea}
\email{bkwon@unist.ac.kr}

\address{(Wanyong Shim) Department of Mathematical Sciences, Korea Advanced Institute of Science and Technology, Daejeon, 34141, Korea}
\email{wyshim25@kaist.ac.kr}


\thanks{\textbf{Acknowledgment.}
B.K. was supported by the National Research Foundation of Korea(NRF) grant funded by the Korea government(MSIT)(00560003). W.S. was supported by Basic Science Research Program through the National Research Foundation of Korea(NRF) funded by the Ministry of Education(RS-2025-25429122).}

\begin{document}

\begin{abstract}
We study the stability of transition fronts for the one-dimensional Cahn--Hilliard equation. More precisely, we prove that for any H\"older initial datum sufficiently close to a front in the $L^2$ norm, the corresponding solution of the Cahn--Hilliard equation exists globally and converges, up to a dynamical shift, to the front as time tends to infinity. The result holds for a class of double-well potentials in a $C^4$-neighborhood of the standard quartic potential. For the proof, we develop an $L^2$-stability framework adapted to Cahn--Hilliard fronts. The key ingredient is a nontrivial second-order Poincar\'e-type inequality that reveals a coercive structure in the indefinite quadratic energy form associated with the linearized operator about the front, once the dynamical shift is taken into account. This yields an $L^2$-contraction estimate and, via a far-field semigroup argument, asymptotic orbital stability of the front in the unweighted $L^2$ topology. The stability analysis requires only the $L^2$-smallness of the initial perturbation; no higher-order smallness, spatial localization, or moment assumptions are imposed.

\end{abstract}
\maketitle


\section{Introduction}\label{sec1}
We consider the one-dimensional Cahn--Hilliard equation
\begin{equation}\label{eq:gCH}
u_t = \left(-u_{xx}+B(u)\right)_{xx}, \quad x\in\mathbb R,\quad t\geq0,
\end{equation}
where \(B=F'\) for a homogeneous free energy density \(F\). Here \(u=u(x,t)\) is a conserved order parameter, representing, for instance, the local concentration difference between two components of a binary mixture. The Cahn--Hilliard equation was introduced by Cahn and Hilliard \cite{CH} as a continuum model for interfacial phenomena in nonuniform binary systems, and has since become a fundamental model for phase separation and spinodal decomposition in binary alloys \cite{Cahn}.

Throughout this paper, we focus on the phase-separation regime determined by a class of double-well potentials in a \(C^4\)-neighborhood of the exact quartic potential
\begin{equation}\label{eq:F} 
F_0(u):=\frac{\kappa^2}{2}\left[\bigg(u-\frac{u_-+u_+}{2}\bigg)^2 - \bigg(\frac{u_+-u_-}{2}\bigg)^2 \right]^2,
\end{equation}
where \(u_-<u_+\) are prescribed constants and \(\kappa>0\) is a scaling parameter. More precisely, for any potential function $F\in C^4(\mathbb R)$, we set $h := F-F_0$ and define
\begin{equation}\label{ppsize}
\rho_F:=\sum_{j=0}^4\|h^{(j)}\|_{L^\infty(\mathbb R)}, \quad h^{(0)}:=h.
\end{equation}
We assume that $\rho_F$ is sufficiently small and that
\begin{equation}\label{p-compat}
h(u_-)=h(u_+)=0,\quad h'(u_-)=h'(u_+)=0.
\end{equation}
For such potentials, $u_-$ and $u_+$ are the two stable phases and remain equal-depth wells of $F$; see Lemma~\ref{lem:front-properties} in Section~\ref{sec:front}. The nonlinear terms corresponding to $F_0$ and $F$ are given by
\begin{equation}\label{eq:B}
B_0(u):=\kappa^2(u_-+u_+-2u)(u-u_-)(u_+-u), \quad B(u):=B_0(u)+h'(u),
\end{equation}
respectively.

Equation \eqref{eq:gCH} has a natural variational structure: it can be written as
\begin{equation*}
u_t=\partial_{xx}\mu, \quad \mu=-u_{xx}+F'(u),
\end{equation*}
where $\mu$ is the chemical potential. Thus \eqref{eq:gCH} is the $H^{-1}$-gradient flow of the one-dimensional Ginzburg--Landau energy
\begin{equation} \label{GL_energy}
E[u] = \int_{\mathbb R} \left( \frac12 |u_x|^2+F(u) \right)\,dx .
\end{equation}
For sufficiently smooth solutions, the energy satisfies the dissipation identity
\begin{equation} \label{GL_dissipation}
\frac{d}{dt}E[u(t)] = -\int_{\mathbb R}|\mu_x|^2\,dx \leq 0.
\end{equation}
At the same time, the equation preserves mass:
\begin{equation} \label{mass}
\frac{d}{dt}\int_{\mathbb R}u(x,t)\,dx = 0,
\end{equation}
whenever the integral is well-defined, or after an appropriate renormalization around the end states. This combination of energy dissipation and mass conservation is one of the central structural features of the Cahn--Hilliard dynamics and is closely related to its role in diffuse-interface motion, coarsening, and pattern formation \cite{EZ, EG, NCS}.

\subsection{Transition front} \label{subsec:front}

For each potential considered here, the Cahn--Hilliard equation \eqref{eq:gCH} admits a one-parameter family of monotone stationary transition fronts connecting the two stable phases. More precisely, for any $x_0\in\mathbb R$, we define the transition front $\bar u$ by
\begin{equation}\label{front} 
x-x_0=\int_{(u_-+u_+)/2}^{\bar u(x)}\frac{1}{\sqrt{2F(y)}} \, dy. 
\end{equation}
Then \(\bar u\) satisfies
\begin{equation}\label{eq:ubar} 
\bar u'=q(\bar u)>0,\quad q(y):=\sqrt{2F(y)},\quad y\in(u_-,u_+), 
\end{equation}
and
\begin{equation} \label{eq:front-endstates}
\lim_{x\to-\infty}\bar u(x)=u_-,\quad \lim_{x\to+\infty}\bar u(x)=u_+. 
\end{equation}
Differentiating the first relation in \eqref{eq:ubar}, we obtain the stationary front equation
\begin{equation} \label{eq:front-ode}
-\bar u''+B(\bar u)=0. 
\end{equation}
For the exact quartic potential $F_0$, the corresponding front is given explicitly by
\begin{equation*}
\bar u_0(x)=\frac{u_-+u_+e^{\kappa L(x-x_0)}}{1+e^{\kappa L(x-x_0)}}=\frac{u_-+u_+}{2}+\frac L2\tanh\left(\frac{\kappa L}{2}(x-x_0)\right), \quad L:=u_+-u_-.
\end{equation*}
The normalized kink is recovered by choosing $F=F_0$, $u_-=-1$, $u_+=1$, and $\kappa=\frac12$. In this case,
\begin{equation*} 
F(u)=\frac18(u^2-1)^2,\quad B(u)=\frac12u^3-\frac12u, 
\end{equation*}
and the corresponding front is, up to translation, given by
\begin{equation*} 
\bar u(x)=\tanh\left(\frac x2\right). 
\end{equation*}
The existence and properties of the front are given in Section~\ref{sec:front}.

\subsection{Background and contributions}
The stability of transition fronts is a natural problem in the analysis of phase-separating systems. For Cahn--Hilliard fronts, however, this problem involves two central difficulties. The first is familiar from the theory of viscous shocks and other translation-invariant coherent structures: the linearized operator about the front has essential spectrum extending up to the origin. Thus, even when spectral stability holds, the neutral mode associated with translation is not separated from the continuous spectrum. This absence of a spectral gap prevents a direct application of the standard semigroup decomposition and also obstructs the usual $L^2$-energy method, since the linearized dissipation does not immediately yield a coercive estimate strong enough to close the nonlinear argument. The second difficulty is more specific to the Cahn--Hilliard dynamics. Although the equation has both a Lyapunov functional and a conservation law, it has neither an $L^1$-contraction property nor a maximum principle. Consequently, one cannot rely on the pointwise or $L^1$-based controls that are often used in scalar parabolic problems to control the displacement of the front.

Despite these difficulties, the stability of the one-dimensional Cahn--Hilliard front has been established by several different methods. For the normalized quartic potential, Bricmont, Kupiainen, and Taskinen \cite{BKT} established a foundational nonlinear stability result in a weighted \(L^\infty\)-type space using an inductive renormalization-group method, obtaining detailed large-time asymptotics. Carlen, Carvalho, and Orlandi \cite{CCO} gave a variational proof for the same potential based on excess free energy and a dissipation dichotomy, assuming \(L^2\)-smallness together with a spatial moment bound. Howard \cite{Ho} studied transition fronts for generalized Cahn--Hilliard equations by a pointwise semigroup method, proving asymptotic orbital stability for algebraically localized perturbations under an Evans-function spectral condition that was later verified in \cite{Ho20}. For further applications of this pointwise semigroup approach to planar and multidimensional fronts, as well as to Cahn--Hilliard systems, we refer to \cite{Ho1,Ho2,Ho3,Ho5,Ho6,HK,HK1,HK2}.

The above results in \cite{BKT,CCO,Ho} address the missing spectral gap through distinct approaches: renormalization-group analysis, variational dissipation estimates, and sharp Green-function bounds, respectively. A common feature of these results is that they are formulated under spatial localization assumptions on perturbations, in the form of weighted norms, moment bounds, or algebraic decay. In particular, in \cite{CCO}, the moment assumption is used to compensate for the lack of an $L^1$-contraction property or a maximum principle. Consequently, the admissible class of initial perturbations is restricted by these localization requirements.

In this context, we establish asymptotic orbital stability directly in the unweighted $L^2$ topology, without imposing any spatial localization or moment assumptions on the initial perturbation. The result holds for the class of double-well potentials in a $C^4$-neighborhood of $F_0$. This is achieved by an \(L^2\)-energy method adapted to the Cahn--Hilliard front. Our analysis identifies a coercive mechanism at the level of the \(L^2\)-energy estimate, yielding an \(L^2\)-contraction bound together with an integrated dissipation estimate. The same dissipation precisely controls the source term in a far-field semigroup formulation based on the reference quartic potential, which in turn yields asymptotic convergence via a semigroup argument for the corresponding far-field operator. This result reveals that the stability of the front is not intrinsically tied to prescribed spatial localization of perturbations; the unweighted \(L^2\) energy framework already contains enough structure to drive convergence to the translation orbit.

We also mention a related relaxation perspective on the one-dimensional Cahn--Hilliard kink. Otto and Westdickenberg \cite{OW} exploited the \(\dot H^{-1}\)-gradient-flow structure of the equation to obtain optimal decay bounds under an energy bound and a bounded homogeneous \(\dot H^{-1}\)-distance to a kink. Otto, Scholtes, and Westdickenberg \cite{OSW} later proved optimal \(L^1\)-type algebraic relaxation rates formulated in terms of first-moment or excess-mass distances. These works are closely related to the stability of the front, but their emphasis is on quantitative relaxation in natural gradient-flow or mass-based distances rather than on the perturbative stability problem discussed here.

The present approach can also be viewed from a broader methodological perspective. In the stability theory of viscous shock profiles, Mascia--Zumbrun and Zumbrun--Howard developed the pointwise semigroup method, based on sharp Green function bounds and Evans-function spectral information, to overcome precisely the lack of a spectral gap \cite{MZ1, MZ2, MZ3, ZH}. This method is powerful and provides refined pointwise and large-time information, but it requires substantial technical work, including detailed spectral analysis and pointwise Green-function estimates. Against this background, a different philosophy has recently emerged in the study of shock stability. In the series of works \cite{KV0,KV,KV1,KV2}, Kang and Vasseur introduced a direct \(L^2\)-energy based approach, using carefully chosen shifts and degenerate Poincar\'e-type inequalities, to obtain contraction and stability estimates for shock profiles. This development showed that the correct coercive structure may be hidden at the energy level and can be recovered by inequalities adapted to the profile.

Our argument follows this philosophy, but the problem considered here is substantially different. The Cahn--Hilliard front is governed by fourth-order, surface-tension-driven dynamics rather than by compression-driven dynamics with second-order viscosity. Thus the coercivity needed for nonlinear stability is not a direct consequence of the usual Poincar\'e inequalities used for viscous shock profiles; it requires instead a second-order degenerate Poincar\'e inequality; see Proposition~\ref{fourth}.

Although weighted Poincar\'e and Hardy--Poincar\'e inequalities for degenerate weights are classical, the estimate used here is more specialized: its coercivity is not apparent from the pointwise structure of the integrand, but is revealed through a Legendre-polynomial decomposition tailored to the reference front. The resulting strict coercivity persists under small perturbations of the reference quartic potential. Consequently, this second-order Poincar\'e-type inequality allows us, within the $L^2$ framework, to address simultaneously the two main difficulties described above: the missing spectral gap and the absence of $L^1$-based control.

\subsection{Notation and function spaces}
Before stating the main theorem, we fix the function-space notation needed below. Throughout the paper, all H\"older spaces consist of bounded functions. For \(\alpha\in(0,\frac12]\), we equip \(C^\alpha(\mathbb R)\) with the norm
\begin{equation*}
\|f\|_{C^\alpha(\mathbb R)}:=\|f\|_{L^\infty(\mathbb R)}+\sup_{x\ne y}\frac{|f(x)-f(y)|}{|x-y|^\alpha}.
\end{equation*}
For \(\Omega=I\times J\subset\mathbb R\times[0,\infty)\), we use the parabolic H\"older norm
\begin{equation*}
\|f\|_{C^{\alpha,\alpha/4}(\Omega)}:=\|f\|_{L^\infty(\Omega)}+\sup_{\substack{x,y\in I,\ x\ne y\\ t\in J}}\frac{|f(x,t)-f(y,t)|}{|x-y|^\alpha}+\sup_{\substack{x\in I\\ t,s\in J,\ t\ne s}}\frac{|f(x,t)-f(x,s)|}{|t-s|^{\alpha/4}}.
\end{equation*}
Finally, \(C^{4+\alpha,1+\alpha/4}(\Omega)\) denotes the usual parabolic H\"older space consisting of bounded functions \(f\) such that \(\partial_x^j f\), \(0\le j\le4\), and \(\partial_t f\) exist classically in \(\Omega\), equipped with the norm
\begin{equation*}
\|f\|_{C^{4+\alpha,1+\alpha/4}(\Omega)}:=\sum_{j=0}^3\|\partial_x^j f\|_{L^\infty(\Omega)}+\|\partial_x^4f\|_{C^{\alpha,\alpha/4}(\Omega)}+\|\partial_t f\|_{C^{\alpha,\alpha/4}(\Omega)}.
\end{equation*}

\subsection{Main result}

We now state the main result of this paper.

\begin{theorem} \label{Main}
Fix constants \(u_-<u_+\), \(\kappa>0\), and \(\alpha\in(0,\frac{1}{2}]\). There exist constants $\rho_0>0$ and $\varepsilon_0>0$ such that the following statement holds.

Let $F\in C^4(\mathbb R)$ satisfy \eqref{p-compat} and $\rho_F\le\rho_0$, where $\rho_F$ is defined in \eqref{ppsize}. Let $\bar u$ be the transition front defined in \eqref{front}. Assume that \(u_0\in C^\alpha(\mathbb R)\) satisfies
\begin{equation*}
\|u_0-\bar u\|_{L^2(\mathbb R)}\le\varepsilon_0.
\end{equation*}
Then the Cahn--Hilliard equation \eqref{eq:gCH} with initial data \(u_0\) admits a unique global-in-time solution \(u\) such that, for any \(0< \sigma < T < \infty \),
\begin{equation} \label{regularity}
u\in C^{\alpha,\alpha/4}(\mathbb R\times[0,T])\cap C^{4+\alpha,1+\alpha/4}(\mathbb R\times[\sigma,T]).
\end{equation}
In particular, \(u\) is a classical solution on \(\mathbb R\times(0,\infty)\). Moreover, there exists a shift function \(\delta\in C^1([0,\infty))\), with \(\delta(0)=0\), such that
\begin{equation} \label{cont}
\lVert u(\cdot+\delta(t),t) - \bar{u}(\cdot) \rVert_{L^2(\mathbb{R})} \leq \lVert u_0 - \bar{u} \rVert_{L^2(\mathbb{R})}
\end{equation}
for all $t \geq 0$. Furthermore,
\begin{equation} \label{L2stability}
\lim_{t \to \infty}{\lVert u(\cdot+\delta(t),t) - \bar{u}(\cdot) \rVert_{L^2(\mathbb{R})}} =0,
\end{equation}
and
\begin{equation} \label{deltalimit}
\lim_{t \to \infty}{|\dot{\delta}(t)|}= 0.
\end{equation}
\end{theorem}

\begin{remark}[On the class of potentials]
The class of potentials considered here is chosen to demonstrate that the coercive mechanism identified in this paper, and the resulting $L^2$-stability framework, are robust under perturbations of the exact quartic potential. In particular, neither the exact quartic algebraic structure nor the corresponding hyperbolic-tangent profile needs to be preserved. The norm $\rho_F$ in \eqref{ppsize} is adopted for convenience: its smallness ensures the uniform estimates used in the proof of Theorem~\ref{Main}, but no optimality is claimed for this choice. The global boundedness of the lower-order derivatives of $h$ is, for example, stronger than is needed in the argument. The endpoint conditions \eqref{p-compat} preserve the locations and equal depths of the wells, but impose neither symmetry nor equal curvatures at $u_-$ and $u_+$. Thus the class includes nonquartic, asymmetric potentials whose fronts need not have the hyperbolic-tangent form.
\end{remark}

\begin{remark}[On the restriction \(\alpha\le \frac{1}{2}\)]
The restriction \(\alpha\le \frac{1}{2}\) does not preclude initial data with higher H\"older regularity. If \(u_0\in C^\gamma(\mathbb R)\) for some \(\gamma\in(\frac{1}{2},1)\), then \(u_0\in C^{1/2}(\mathbb R)\), so Theorem~\ref{Main} applies with \(\alpha=\frac{1}{2}\). The upper bound on \(\alpha\) enters only in the continuation argument of Section~\ref{sec2.1}, where the a priori estimate gives a uniform $H^1$-bound on the perturbation. By the one-dimensional Sobolev embedding, this yields \(C^{1/2}\)-control, and hence the continuation argument does not propagate H\"older regularity with exponent greater than $\frac{1}{2}$.
\end{remark}

\begin{remark}[Admissible solution classes]
The assumption \(u_0\in C^\alpha(\mathbb R)\) in Theorem~\ref{Main} is used to invoke the short-time existence theory for classical solutions of nonlinear parabolic equations in divergence form, following \cite{Ho4}. The stability argument itself does not rely on this specific H\"older regularity, but only on the validity of the construction of the shift, the density argument in the a priori estimate, and the Duhamel representation used in the far-field semigroup argument; cf. Appendix~\ref{app:L2-framework}. Consequently, for any global solution belonging to such an \(L^2\)-admissible class, the conclusions \eqref{cont}--\eqref{deltalimit} also hold.
\end{remark}

\begin{remark}[On the limiting front location]
The convergence in \eqref{L2stability} implies that the transition front $\bar u$ is asymptotically orbitally stable, in the sense that the solution converges asymptotically to the translation orbit of the front. This statement, however, does not identify a limiting translate. From \eqref{deltalimit}, together with \(\delta(0)=0\), it follows that
\begin{equation*}
\frac{\delta(t)}{t}=\frac{1}{t}\int_0^t\dot\delta(s)\,ds \to 0 \quad \text{as }t\to\infty.
\end{equation*}
Thus the front location may drift sublinearly in time. This is a natural consequence of the present \(L^2\) framework: a limiting translate is typically identified through conservation of the relative mass; cf. \cite{BKT, CCO, Ho, OSW}. If the relative mass
\begin{equation*}
m(t):=\int_{\mathbb R}(u(x,t)-\bar u(x))\,dx
\end{equation*}
is well defined, then mass conservation \eqref{mass} yields \(m(t)=m(0)\), and convergence to a fixed translate \(\bar u(\cdot-\delta_\infty)\) would formally impose
\begin{equation*}
m(0)=\int_{\mathbb R}(\bar u(x-\delta_\infty)-\bar u(x))\,dx=-(u_+-u_-)\delta_\infty.
\end{equation*}
For a merely \(L^2\)-small perturbation, however, this relative mass need not be finite, and the corresponding mass-selection mechanism is not available. Consequently, identifying a limiting translate generally requires additional assumptions, such as \(L^1\)-localization or suitable moment conditions on the initial perturbation.
\end{remark}

\subsection{Proof strategy and organization of the paper}
The proof proceeds in two main stages. First, we derive an a priori estimate for a suitably shifted perturbation and use it to obtain the global $L^2$-contraction bound \eqref{cont}. Second, we combine the resulting dissipation estimate with a far-field semigroup argument to prove asymptotic stability.

The main point in the first stage is to identify the coercive mechanism behind the $L^2$-energy estimate. We define the perturbation $\phi(x,t):=u(x+\delta(t),t)-\bar u(x)$ and the linearized operator
\begin{equation*}
\mathcal L\phi:=\partial_x^2(-\phi_{xx}+B'(\bar u)\phi).
\end{equation*}
The associated quadratic energy form $-\langle \phi, \mathcal{L}\phi \rangle$ is not coercive by itself, due to the translational mode. The coercive estimate is obtained by combining this form with the damping term generated by the shift and rewriting the resulting expression in the front coordinate $y=\bar u(x)$.

The key observation is that the front coordinate \(y=\bar u(x)\) performs two tasks at once. It compactifies the whole-line problem to the finite interval \([u_-,u_+]\), and, for the exact quartic potential, it transforms the linearized energy form into a weighted form with Legendre structure. More precisely, letting \(\Phi(y)=\phi(x)\) and \(q(y)=\bar u'(x)=\sqrt{2F(y)}\), the differential part of the linearized energy becomes a degenerate quadratic form involving \(q^3\Phi_{yy}^2\) and a zeroth-order potential; see Section~\ref{sec2.3.2}. For $F=F_0$, the normalization of \([u_-,u_+]\) to \([-1,1]\) transforms $q_0(y)=\kappa(y-u_-)(u_+-y)$ into a multiple of \(1-z^2\), the coefficient appearing in the Legendre operator. This allows us to use the Legendre polynomial expansion to estimate the compactified energy form; see Section~\ref{sec4.1} and Appendix~\ref{app:Q0-coefficients}.

This reformulation reveals, at the \(L^2\)-energy level, the obstruction to coercivity caused by translation invariance. In the normalized setting, the translational direction is proportional to \(1-z^2\), hence it is a combination of low-order Legendre modes and has nonzero mean. The shift is then chosen by \(\dot\delta(t)=-\eta\int_{\mathbb R}\bar u'\phi\,dx\), with \(\eta=4\kappa^3L\), so that the energy identity contains the term \(\eta(\int_{\mathbb R}\bar u'\phi\,dx)^2\). In the compactified variable, this becomes the rank-one damping term \(\eta(\int_{u_-}^{u_+}\Phi \,dy)^2\) acting on the mean, thereby removing the mean-mode degeneracy. Writing the normalized function as \(\Psi\), the resulting coercive estimate for $F=F_0$ is given by the following front-adapted Poincar\'e-type inequality:
\begin{equation*}
\begin{split}
& \frac18\int_{-1}^{1}(1-z^2)^3 \Psi_{zz}(z)^2\,dz +\frac34\int_{-1}^{1}(3z^2-1)\Psi(z)^2\,dz + \left( \int_{-1}^{1}\Psi(z)\,dz \right)^2 \\
& \quad \geq \frac{3}{56} \left( \int_{-1}^{1}(1-z^2)\Psi_z(z)^2\,dz + \left( \int_{-1}^{1}\Psi(z)\,dz \right)^2 \right) \\
& \quad \geq \frac{3}{28}\int_{-1}^{1}\Psi(z)^2\,dz.
\end{split}
\end{equation*}

To extend this coercivity to the perturbed potentials, we compare the compactified quadratic forms associated with $F$ and $F_0$ on the fixed interval $(u_-,u_+)$. When evaluated at the same function $\Phi$, their difference is controlled by $\rho_F$ times the coercive energy for $F_0$. This error can therefore be absorbed for sufficiently small $\rho_F$, yielding the coercivity estimate in Proposition~\ref{fourth}; see Section~\ref{sec:potential-perturbation}.

Once this coercivity is available, the nonlinear terms are absorbed by the smallness of the \(L^2\)-perturbation, giving the global \(L^2\)-contraction estimate and an integrated dissipation bound. The asymptotic convergence is proved by a separate semigroup argument. The perturbation equation is rewritten around the damped reference far-field operator \(\mathcal A_\infty\); its homogeneous semigroup is strongly stable, while the inhomogeneous source is controlled by the dissipation estimate obtained above. This proves asymptotic orbital stability in $L^2$.

The paper is organized as follows. Section~\ref{sec:front} proves the existence and properties of the transition fronts. Section~\ref{sec2} establishes global existence and the \(L^2\)-contraction estimate based on the front-adapted Poincar\'e-type inequalities. Section~\ref{sec3} proves asymptotic stability through the far-field semigroup formulation. Section~\ref{sec4} proves the Poincar\'e-type inequalities via the normalized Legendre formulation of the compactified energy form for the exact quartic potential, and proves their persistence under perturbations of the potential.

\section{Existence and properties of the front} \label{sec:front}

We prove the existence of the transition fronts introduced in Section~\ref{subsec:front} and derive the estimates needed in the stability analysis. These estimates hold uniformly for potentials satisfying \eqref{p-compat} with $\rho_F$ sufficiently small.

\begin{lemma} \label{lem:front-properties}
Let $F\in C^4(\mathbb R)$ satisfy \eqref{p-compat}. There exist constants $\rho_1>0$ and $C>0$, depending only on $\kappa$ and $L$, such that the following holds.

If $\rho_F\le\rho_1$, then $F$ satisfies
\begin{equation} \label{eq:potential-comparison}
\frac12 F_0(u) \le F(u)\le\frac32 F_0(u),\quad 6\kappa^2 \le F^{(4)}(u)\le18\kappa^2,\quad u \in \mathbb R,
\end{equation}
and
\begin{equation} \label{eq:potential-wells}
F(u_\pm)=F'(u_\pm)=0,\quad F''(u_\pm) >0.
\end{equation}
Moreover,
\begin{equation} \label{eq:profile-coefficient-closeness}
\left\|\frac q{q_0}-1\right\|_{L^\infty(u_-,u_+)}+\|q_y-(q_0)_y\|_{L^\infty(u_-,u_+)}\le C\rho_F,
\end{equation}
where
\begin{equation*}
q(y):=\sqrt{2F(y)},\quad q_0(y):=\kappa(y-u_-)(u_+-y).
\end{equation*}
For each $x_0\in\mathbb R$, there exists a unique stationary front $\bar u$ satisfying \eqref{eq:front-ode} and \eqref{eq:front-endstates}, with $\bar u(x_0)=\frac{u_-+u_+}{2}$. The front $\bar{u}$ is given implicitly by \eqref{front} and satisfies
\begin{equation} \label{eq:front-bounds}
0<\bar u'(x)\le C,\quad |\bar u''(x)|\le C\bar u'(x),\quad x\in\mathbb R,
\end{equation}
and
\begin{equation} \label{eq:front-integrability}
\|\bar u'\|_{L^1(\mathbb R)}=L,\quad \|\bar u'\|_{L^2(\mathbb R)}^2+\|\bar u''\|_{L^1(\mathbb R)}\le C.
\end{equation}
\end{lemma}

\begin{proof}
We first prove \eqref{eq:potential-comparison} and \eqref{eq:potential-wells}. Fix $u\in\mathbb R$, and choose $u_*\in\{u_-,u_+\}$ such that
\begin{equation*}
|u-u_*|=\min\{|u-u_-|,|u-u_+|\}.
\end{equation*}
By \eqref{p-compat} and Taylor's theorem, we have
\begin{equation*}
\begin{split}
F(u)-F_0(u)
&=h(u)-h(u_*)-h'(u_*)(u-u_*)\\
&=(u-u_*)^2\int_0^1(1-s)h''(u_*+s(u-u_*))\,ds.
\end{split}
\end{equation*}
Using $\|h''\|_{L^\infty(\mathbb R)}\le\rho_F$ from \eqref{ppsize}, we obtain
\begin{equation*}
|F(u)-F_0(u)|\le |u-u_*|^2 \lVert h'' \rVert_{L^\infty} \int_0^1 (1-s) \, ds \leq \frac12 \rho_F \min\{|u-u_-|^2,|u-u_+|^2\}.
\end{equation*}
Since
\begin{equation*}
\min\{|u-u_-|^2,|u-u_+|^2\} = \frac{(u-u_-)^2(u-u_+)^2}{\max\{|u-u_-|^2,|u-u_+|^2\}} \le \frac{4}{L^2}(u-u_-)^2(u-u_+)^2,
\end{equation*}
we arrive at
\begin{equation}\label{eq:relative-potential-bound}
|F(u)-F_0(u)| \le\frac{2\rho_F}{L^2}(u-u_-)^2(u-u_+)^2 =\frac{4\rho_F}{\kappa^2L^2}F_0(u),\quad u\in\mathbb R.
\end{equation}
Also, by \eqref{eq:F}, \eqref{ppsize}, and \eqref{p-compat},
\begin{equation*}
\begin{split}
&|F^{(4)}(u)-12\kappa^2| = |h^{(4)}(u)| \le \rho_F, \quad u \in \mathbb R,\\
& F(u_\pm) = F'(u_\pm)=0,\quad F''(u_\pm)=\kappa^2L^2+h''(u_\pm) \ge \kappa^2L^2-\rho_F.
\end{split}
\end{equation*}
Thus \eqref{eq:potential-comparison} and \eqref{eq:potential-wells} hold for sufficiently small $\rho_1>0$.

We next prove \eqref{eq:profile-coefficient-closeness}. Since $q_0^2=2F_0$ on $(u_-,u_+)$, \eqref{eq:relative-potential-bound} gives
\begin{equation} \label{qq_0bd}
\left|\frac q{q_0}-1\right|=\left|\sqrt{\frac F{F_0}}-1\right|=\frac{|F-F_0|}{F_0(\sqrt{F/F_0}+1)}\le C\rho_F.
\end{equation}
Moreover, by \eqref{eq:potential-comparison},
\begin{equation} \label{eq:q-comparison}
\frac{1}{\sqrt2}q_0(y)\le q(y)\le\sqrt{\frac32}\,q_0(y),\quad y\in(u_-,u_+).
\end{equation}
For the derivative, using $qq_y=F'$ and $q_0(q_0)_y=F_0'$, we compute
\begin{equation*}
q_y-(q_0)_y=\frac{F'}q-\frac{F_0'}{q_0}=\frac{h'}q+(q_0)_y\left(\frac{q_0}q-1\right).
\end{equation*}
Since $h'(u_\pm)=0$ and $|h''| \leq \rho_F$, we have
\begin{equation*}
|h'(y)| = \left| \int_{u_\pm}^y h''(s)\,ds \right| \le \rho_F\min\{y-u_-,u_+-y\}.
\end{equation*}
This, together with \eqref{qq_0bd}, \eqref{eq:q-comparison}, $|(q_0)_y|\le\kappa L$, and $\max\{y-u_-,u_+-y\}\ge \frac{L}{2}$, yields
\begin{equation} \label{qq0ybd}
\begin{split}
|q_y-(q_0)_y|&\le\frac{|h'|}{q}+|(q_0)_y|\frac{q_0}{q}\left|\frac q{q_0}-1\right|\\
&\le\frac{C\rho_F}{\kappa\max\{y-u_-,u_+-y\}}+C\rho_F\le C\rho_F.
\end{split}
\end{equation}
Combining \eqref{qq_0bd} and \eqref{qq0ybd} proves \eqref{eq:profile-coefficient-closeness}.

We now prove the existence of the front. By \eqref{eq:potential-comparison}, the function $F''$ is strictly convex. Since
\begin{equation*}
F''(u_\pm)>0,\quad \int_{u_-}^{u_+}F''(y)\,dy=F'(u_+)-F'(u_-)=0,
\end{equation*}
the function $F''$ has exactly two zeros in $(u_-,u_+)$, and $F'$ has exactly one zero between them. Hence $F$ has the double-well form in \cite[Definition~2.1]{Ho1}. Together with \eqref{eq:potential-wells}, this verifies the assumptions of \cite[Lemma~2.1]{Ho1}, which yields an increasing stationary front $\bar{u}$ connecting $u_-$ to $u_+$, unique up to translation. The condition $\bar u(x_0)=\frac{u_-+u_+}{2}$ fixes the translation and hence yields uniqueness. Multiplying the front equation \eqref{eq:front-ode} by $\bar u'$ and integrating from $-\infty$ to $x$, we obtain
\begin{equation*}
\frac12\bar u'(x)^2=F(\bar u(x)),\quad \bar u'(x)=q(\bar u(x))>0.
\end{equation*}
The implicit formula \eqref{front} then follows by dividing by $q(\bar u)$ and integrating from $x_0$ to $x$.

It remains to establish \eqref{eq:front-bounds} and \eqref{eq:front-integrability}. By \eqref{eq:q-comparison} and the bound $q_0(y)\le \frac{\kappa L^2}{4}$, we have
\begin{equation*}
0<\bar u'(x)=q(\bar u(x))\le Cq_0(\bar u(x))\le C.
\end{equation*}
Moreover, by \eqref{qq0ybd} and the bound $|(q_0)_y|\le\kappa L$,
\begin{equation*}
|\bar u''(x)|=|q_y(\bar u(x))|\bar u'(x)
\le (\kappa L+C\rho_F)\bar u'(x)
\le C\bar u'(x).
\end{equation*}
These estimates prove \eqref{eq:front-bounds}. By monotonicity and \eqref{eq:front-endstates},
\begin{equation*}
\|\bar u'\|_{L^1(\mathbb R)}=\int_{\mathbb R}\bar u'(x)\,dx=[\bar u(x)]_{-\infty}^{+\infty}=L.
\end{equation*}
Using \eqref{eq:front-bounds}, we obtain
\begin{equation*}
\begin{split}
\|\bar u'\|_{L^2(\mathbb R)}^2 \le\|\bar u'\|_{L^\infty(\mathbb R)}\|\bar u'\|_{L^1(\mathbb R)}\le C, \quad \|\bar u''\|_{L^1(\mathbb R)} \le C\|\bar u'\|_{L^1(\mathbb R)}\le C.
\end{split}
\end{equation*}
This completes the proof.
\end{proof}

\section{\texorpdfstring{$L^2$-contraction property}{L2-contraction property}} \label{sec2}
In this section, we establish the global $L^2$-contraction bound \eqref{cont} in Theorem~\ref{Main}. We first present a local a priori contraction estimate under an $L^2$-smallness assumption, with the proof deferred to Section~\ref{sec2.3}. Combining this estimate with the short-time existence theory and a standard continuation argument, we prove that the resulting classical solution extends globally in time and satisfies the contraction estimate for all times.

We define the perturbation around the front $\bar{u}$ by
\begin{equation} \label{pert}
\phi(x,t) := u(x+\delta(t),t)-\bar{u}(x),
\end{equation}
where $u(x,t)$ is a solution of the Cahn--Hilliard equation \eqref{eq:gCH}, and $\delta(t)$ is a time-dependent shift function to be determined. Then $\phi$ satisfies the perturbation equation:
\begin{equation} \label{eq:pert}
\phi_t = \left(-\phi_{xx}+B'(\bar{u})\phi\right)_{xx} + \left(B(\bar{u}+\phi)-B(\bar{u})-B'(\bar{u})\phi\right)_{xx} + \dot{\delta}(t)\left(\bar{u}' + \phi_x\right).
\end{equation}
We define the shift function $\delta(t)$ as the solution of the ODE:
\begin{equation} \label{delta}
\dot{\delta}(t) = - \eta \int_\mathbb{R} \bar{u}'\phi \, dx, \quad \delta(0)=0,
\end{equation}
where $\textstyle \eta := 4\kappa^3 L >0$. For any bounded solution $u$, the integrability $\bar u',\bar u''\in L^1(\mathbb R)$ from \eqref{eq:front-integrability} implies that the right-hand side of \eqref{delta} is well-defined and Lipschitz continuous in $\delta$. Indeed, writing $G(t,\delta)$ for the right-hand side of \eqref{delta}, the change of variables $z=x+\delta$ gives
\begin{equation*}
|G(t,\delta_1)-G(t,\delta_2)|\le \eta\|u(\cdot,t)\|_{L^\infty}\|\bar u''\|_{L^1}|\delta_1-\delta_2|.
\end{equation*}
Hence the scalar ODE \eqref{delta} admits a $C^1$ solution, at least locally in time; the a priori estimate below extends it to all $t\geq0$.

\subsection{A priori estimate and global continuation} \label{sec2.1}

We first specify the regularity class in which the a priori estimate is formulated.

\begin{definition}[Admissible perturbations] \label{def:admissible}
Let $T_0>0$, and let $\phi$ be a perturbation defined by \eqref{pert}, with $\delta$ defined by \eqref{delta}. We say that the perturbation \(\phi\) is admissible on \([0,T_0]\) if
\begin{equation*}
\phi\in C([0,T_0];L^2(\mathbb R)), \quad \delta\in C^1([0,T_0]),
\end{equation*}
and the underlying solution \(u\) satisfies
\begin{equation*}
u \in C^{\alpha,\alpha/4}(\mathbb{R} \times [0,T_0]) \cap C^{4+\alpha,1+\alpha/4}(\mathbb{R}\times [\sigma,T_0])
\end{equation*}
for any $\sigma \in (0,T_0)$.
\end{definition}

We state the local a priori estimate.

\begin{proposition}[A priori contraction estimate] \label{apriori}
There exist constants $\varepsilon_1>0$ and $c_0>0$ such that the following statement holds.

Let $T_0>0$, and let $\phi$ be an admissible perturbation on $[0,T_0]$ in the sense of Definition~\ref{def:admissible}, with $\delta$ defined by \eqref{delta}. If
\begin{equation*}
\sup_{0 \leq t \leq T_0} \lVert \phi(\cdot,t) \rVert_{L^2(\mathbb{R})} \leq \varepsilon_1,
\end{equation*}
then
\begin{equation} \label{aprioriest}
\lVert \phi(\cdot,t) \rVert_{L^2(\mathbb{R})}^2+c_0\int_0^t \mathcal{D}(s)\,ds\leq \lVert \phi(\cdot,0) \rVert_{L^2(\mathbb{R})}^2
\end{equation}
for all $t \in [0,T_0]$, where
\begin{equation*}
\mathcal{D}(s):=\lVert \phi_x(\cdot,s) \rVert_{H^1(\mathbb{R})}^2+\lVert(\sqrt{\bar u'}\phi)(\cdot,s)\rVert_{L^2(\mathbb{R})}^2+|\dot\delta(s)|^2.
\end{equation*}
\end{proposition}

We now prove the regularity assertion and the global contraction bound in Theorem~\ref{Main} using Proposition~\ref{apriori}.

\begin{proof}[Proof of \eqref{regularity} and \eqref{cont}]

By the short-time existence result for \(C^\alpha\cap L^2\) perturbations in Appendix~\ref{app:L2-persistence}, there exists \(T>0\) such that the solution $u$ with initial data \(u_0\) is defined on \([0,T]\) and its associated perturbation $\phi$ is admissible on \([0,T]\). Let \(T_{\max}\in(0,\infty]\) be the supremum of all such \(T>0\). Then the solution is defined on \([0,T_{\max})\). Moreover, for all \(T<T_{\max}\), the perturbation \(\phi\) is admissible on \([0,T]\), and \(u\) is classical on \(\mathbb R\times(0,T_{\max})\).

Choose $\varepsilon_0<\frac{\varepsilon_1}{2}$ and assume
\begin{equation*}
\lVert \phi (\cdot,0) \rVert_{L^2(\mathbb{R})} = \lVert u(\cdot,0)-\bar{u}\rVert_{L^2(\mathbb{R})} \leq \varepsilon_0.
\end{equation*}
Since $\phi \in C([0,T];L^2(\mathbb{R}))$ for any $T < T_{\max}$, there exists a time interval $[0,\tau_0] \subset [0,T_{\max})$ such that
\begin{equation*}
\sup_{0 \leq t \leq \tau_0}{\lVert \phi (\cdot,t) \rVert_{L^2(\mathbb{R})}} \leq \varepsilon_1.
\end{equation*}
The a priori estimate of Proposition~\ref{apriori} then applies on $[0,\tau_0]$, so that
\begin{equation*}
\sup_{0 \leq t \leq \tau_0}{\lVert \phi (\cdot,t) \rVert_{L^2(\mathbb{R})}} \leq \lVert \phi(\cdot,0) \rVert_{L^2(\mathbb{R})} \leq \varepsilon_0.
\end{equation*}
Thus, by a standard continuation argument, the a priori $L^2$-contraction estimate \eqref{aprioriest} remains valid up to any time $t<T_{\max}$.

To complete the proof, it remains to show \(T_{\max}=\infty\). Suppose, for contradiction, that \(T_{\max}<\infty\). By \eqref{aprioriest} and the definition of \(\mathcal D\), we may choose \(t_*\in(0,T_{\max})\) such that \(\phi(\cdot,t_*)\in H^2(\mathbb R)\). Set $\phi_*:=\phi(\cdot,t_*)$. Since \(u(\cdot+\delta(t_*),t_*)=\bar u+\phi_*\) and the Ginzburg--Landau energy is invariant under spatial translations, \eqref{GL_energy} gives
\begin{equation*}
E[u(t_*)]=\frac12\int_{\mathbb R}|(\phi_*)_x+\bar u'|^2\,dx+\int_{\mathbb R}F(\bar u+\phi_*)\,dx=:I+II.
\end{equation*}
The first term is estimated as
\begin{equation*}
I\le C\left(\|(\phi_*)_x\|_{L^2(\mathbb R)}^2+\|\bar u'\|_{L^2(\mathbb R)}^2\right)<\infty.
\end{equation*}
By \eqref{eq:potential-comparison} and Young's inequality, we have
\begin{equation*}
F(\bar{u} + \phi_*) \leq C (|(\bar u -u_-)(\bar u -u_+)|^2 + |\phi_*|^2 + |\phi_*|^3 + |\phi_*|^4).
\end{equation*}
Moreover, by \eqref{eq:potential-comparison}, \eqref{eq:ubar}, and \eqref{eq:front-integrability},
\begin{equation*}
\begin{split}
\int_{\mathbb R}|(\bar u-u_-)(\bar u-u_+)|^2\,dx&=\frac{2}{\kappa^2}\int_{\mathbb R}F_0(\bar u)\,dx\\
&\le\frac{4}{\kappa^2}\int_{\mathbb R}F(\bar u)\,dx=\frac{2}{\kappa^2}\|\bar u'\|_{L^2(\mathbb R)}^2\le C.
\end{split}
\end{equation*}
These estimates, together with $\lVert \phi_* \rVert_{L^\infty} \leq C \lVert \phi_* \rVert_{H^1}$, give
\begin{equation*}
II\le C_{\bar u} + C\left(\|\phi_*\|_{L^\infty(\mathbb R)}^2\|\phi_*\|_{L^2(\mathbb R)}^2+\|\phi_*\|_{L^\infty(\mathbb R)}\|\phi_*\|_{L^2(\mathbb R)}^2+\|\phi_*\|_{L^2(\mathbb R)}^2 \right) < \infty.
\end{equation*}
Hence $E[u(t_*)]$ is finite. Since $u$ is a classical solution on $(0,T_{\max})$, by \eqref{GL_dissipation},
\begin{equation*}
E[u(t)]\le E[u(t_*)]\quad\text{for }t_*\le t<T_{\max}.
\end{equation*}
Moreover, using $F\ge0$, we obtain
\begin{equation*}
\frac12\|u_x(\cdot,t)\|_{L^2(\mathbb R)}^2\le E[u(t)]\le E[u(t_*)]\quad\text{for }t_*\le t<T_{\max}.
\end{equation*}
Therefore,
\begin{equation*}
\|\phi_x(\cdot,t)\|_{L^2(\mathbb R)} \le \|u_x(\cdot+\delta(t),t)\|_{L^2(\mathbb R)}+\|\bar u'\|_{L^2(\mathbb R)} \le \sqrt{2E[u(t_*)]}+\|\bar u'\|_{L^2(\mathbb R)}
\end{equation*}
for $t_*\le t<T_{\max}$. Together with the $L^2$-bound in \eqref{aprioriest}, this gives
\begin{equation*}
\sup_{t_*\le t<T_{\max}}\|\phi(\cdot,t)\|_{H^1(\mathbb R)}<\infty.
\end{equation*}
By the Sobolev embedding $H^1(\mathbb R)\hookrightarrow C^{1/2}(\mathbb R)$ and translation invariance of H\"older norms, $u(\cdot,t)$ remains uniformly bounded in $C^{1/2}(\mathbb R)$ as $t\uparrow T_{\max}$. The short-time existence result can therefore be reapplied at times arbitrarily close to $T_{\max}$, yielding a continuation of the solution beyond $T_{\max}$. This contradicts the maximality of $T_{\max}$, and hence $T_{\max}=\infty$. Finally, the regularity stated in \eqref{regularity} follows from the admissibility in Definition~\ref{def:admissible}.
\end{proof}

\subsection{Poincar\'e-type inequalities} \label{sec2.2}
Before proving Proposition~\ref{apriori}, we record two Poincar\'e-type inequalities and a useful consequence. These estimates will be used to obtain a coercive bound for the linearized energy form around the front. We define
\begin{equation} \label{qV}
q(y) := \sqrt{2F(y)}, \quad V(y) := - \frac{B''(y)B(y)+B'''(y)q(y)^2}{2q(y)} \quad \text{for }y \in (u_-,u_+),
\end{equation}
where $B(\cdot)$ is as in \eqref{eq:B}. Note that $B=qq_y$ gives $V=-\frac12(B''q_y+B'''q)$, which is bounded on $(u_-,u_+)$ by \eqref{ppsize} and \eqref{eq:profile-coefficient-closeness}.

\begin{proposition} \label{fourth}
There exist constants $\rho_0>0$ and $c_{\kappa,L}>0$ such that, if $F$ satisfies \eqref{p-compat} and $\rho_F\le\rho_0$, then, for any $\Phi\in L^2(u_-,u_+)$ satisfying
\begin{equation*}
q^{1/2}\Phi_y \in L^2(u_-,u_+), \quad q^{3/2}\Phi_{yy} \in L^2(u_-,u_+),
\end{equation*}
we have
\begin{equation} \label{fourth1}
\begin{split}
& \int_{u_-}^{u_+} q(y)^3\Phi_{yy}(y)^2\,dy+\int_{u_-}^{u_+}V(y)\Phi(y)^2\,dy + \eta \left( \int_{u_-}^{u_+}\Phi(y)\,dy \right)^2 \\
& \qquad \geq c_{\kappa,L} \left( \int_{u_-}^{u_+} q(y) \Phi_y(y)^2 \, dy + \left( \int_{u_-}^{u_+} \Phi(y) \, dy \right)^2 \right),
\end{split}
\end{equation}
where $\eta = 4\kappa^3L$. Moreover, there exists a constant $C_{\kappa,L}>0$ such that
\begin{equation} \label{fourth2}
\int_{u_-}^{u_+}\Phi(y)^2\,dy \le C_{\kappa,L} \left( \int_{u_-}^{u_+}q(y)\Phi_y(y)^2\,dy + \left( \int_{u_-}^{u_+}\Phi(y)\,dy \right)^2 \right).
\end{equation}
\end{proposition}

\begin{corollary} \label{fourth-cor}
For $\Phi$ as in Proposition~\ref{fourth}, there exists a constant $\widetilde{C}_{\kappa,L}>0$ such that
\begin{equation} \label{fourth3}
\begin{split}
& \int_{u_-}^{u_+}q(y)\Phi_y(y)^2\,dy+\int_{u_-}^{u_+}q(y)^3\Phi_{yy}(y)^2 \,dy  \\
& \qquad \leq \widetilde{C}_{\kappa,L} \left( \int_{u_-}^{u_+} q(y)^3\Phi_{yy}(y)^2\,dy+\int_{u_-}^{u_+}V(y)\Phi(y)^2\,dy+ \eta \left( \int_{u_-}^{u_+}\Phi(y)\,dy \right)^2 \right).
\end{split}
\end{equation}
\end{corollary}

The proofs of Proposition~\ref{fourth} and Corollary~\ref{fourth-cor} are deferred to Section~\ref{sec4}.

\subsection{Proof of Proposition~\ref{apriori}} \label{sec2.3}

We prove Proposition~\ref{apriori} by deriving an energy identity, estimating the linear part via the Poincar\'e-type inequalities above, and controlling the nonlinear part using the smallness assumption. We carry out the estimates for smooth perturbations generated by compactly supported initial perturbations. The a priori estimate \eqref{aprioriest} for admissible perturbations then holds by a density argument; see Appendix~\ref{app:smooth-approx}.

\subsubsection{Energy identity} 
Multiplying \eqref{eq:pert} by \(\phi\) and integrating over $\mathbb R$, we obtain
\begin{equation*}
\begin{split}
\frac{1}{2}\frac{d}{dt}\int_{\mathbb R}\phi^2\,dx
&= \int_{\mathbb R}\phi\left(-\phi_{xx}+B'(\bar{u})\phi\right)_{xx}\,dx  + \int_{\mathbb R}\phi\left(B(\bar{u}+\phi)-B(\bar{u})-B'(\bar{u})\phi\right)_{xx}\,dx \\
&\quad + \dot{\delta}(t)\int_{\mathbb R}\phi(\bar{u}'+\phi_x)\,dx.
\end{split}
\end{equation*}
Integrating by parts, we compute the first term on the right-hand side as
\begin{equation*}
\begin{split}
\int_{\mathbb R} \phi \left( -\phi_{xx}+B'(\bar{u})\phi \right)_{xx} \, dx
&= - \int_{\mathbb R} \left( \phi_{xx}^2 + B'(\bar{u})\phi_x^2-\frac{1}{2}\big( B'(\bar{u}) \big)_{xx} \phi^2\right) \, dx.
\end{split}
\end{equation*}
Therefore, setting $\delta(t)$ as in \eqref{delta}, we have
\begin{equation}\label{E_identity}
\frac{1}{2}\frac{d}{dt} \int_{\mathbb R}\phi^2\,dx = -\mathcal{Q}(t) +\mathcal{N}(t),
\end{equation}
where
\begin{equation*}
\begin{split}
\mathcal{Q}(t)& :=\int_{\mathbb R}\left(\phi_{xx}^2+B'(\bar{u})\phi_x^2-\frac{1}{2}\big( B'(\bar{u}) \big)_{xx} \phi^2 \right)\,dx + \eta \left( \int_\mathbb{R} \bar{u}'\phi \, dx \right)^2, \\
\mathcal{N}(t) & :=\int_{\mathbb R}\phi\left(B(\bar{u}+\phi)-B(\bar{u})-B'(\bar{u})\phi\right)_{xx}\,dx.
\end{split}
\end{equation*}
Now we estimate $\mathcal{Q}(t)$ and $\mathcal{N}(t)$ separately.

\subsubsection{Estimate of the linear part} \label{sec2.3.2}
We first estimate $\mathcal{Q}(t)$. Let \(y=\bar{u}(x)\) and \(\Phi(y,t)=\phi(x,t)\). Using \eqref{eq:ubar} and the relations
\begin{equation*}
\phi_x=q(y)\Phi_y,\quad \phi_{xx}=q(y)q_y(y)\Phi_y+q(y)^2\Phi_{yy}=B(y)\Phi_y+q(y)^2\Phi_{yy},
\end{equation*}
we obtain
\begin{equation*}
\mathcal{Q}(t)=\int_{u_-}^{u_+}\left( q^3\Phi_{yy}^2+\left(\frac{B^2}{q}+B'q\right)\Phi_y^2+2Bq\Phi_y\Phi_{yy}+V\Phi^2\right)\,dy+\eta\left(\int_{u_-}^{u_+}\Phi\,dy\right)^2.
\end{equation*}
By integration by parts,
\begin{equation*}
\int_{u_-}^{u_+}2Bq\Phi_y\Phi_{yy}\,dy=[Bq\Phi_y^2]_{u_-}^{u_+}-\int_{u_-}^{u_+}(Bq)'\Phi_y^2\,dy.
\end{equation*}
Here, the endpoint term vanishes for the smooth approximations under consideration; see Appendix~\ref{app:front-coordinate}. Since $(Bq)'=B'q+\frac{B^2}{q}$, we arrive at
\begin{equation*}
\mathcal{Q}(t)=\int_{u_-}^{u_+}q(y)^3\Phi_{yy}(y,t)^2\,dy+\int_{u_-}^{u_+}V(y)\Phi(y,t)^2\,dy+\eta\left(\int_{u_-}^{u_+}\Phi(y,t)\,dy\right)^2,
\end{equation*}
where we also used
\begin{equation*}
\big( B'(\bar{u}) \big)_{xx}=B''(y)B(y)+B'''(y)q(y)^2.
\end{equation*}
By \eqref{eq:profile-coefficient-closeness}, the corresponding identity in the normalized weighted space $X$ follows from the polynomial density result in Appendix~\ref{density}.

By Proposition~\ref{fourth}, there exist constants $c_{\kappa,L}>0$ and $\tilde{c}_{\kappa,L}>0$ such that
\begin{equation*}
\begin{split}
\mathcal{Q}(t) & \geq c_{\kappa,L} \left( \int_{u_-}^{u_+} q(y) \Phi_y(y,t)^2 \, dy +  \left( \int_{u_-}^{u_+} \Phi(y,t) \, dy \right)^2 \right) \\
&  \geq \tilde{c}_{\kappa,L} \int_{u_-}^{u_+} \Phi(y,t)^2 \, dy.
\end{split}
\end{equation*}
Thus, there exists $C>0$ such that
\begin{equation} \label{coerc1}
\lVert \phi_x(\cdot,t) \rVert_{L^2}^2, \ |\dot{\delta}(t)|^2, \ \lVert (\sqrt{\bar{u}'}\phi)(\cdot,t) \rVert_{L^2}^2 \leq C \mathcal{Q}(t).
\end{equation}
Moreover, since
\begin{equation*}
\phi_{xx}=q(y)q_y(y)\Phi_y+q(y)^2\Phi_{yy}=B(y)\Phi_y+q(y)^2\Phi_{yy},
\end{equation*}
we have
\begin{equation} \label{coerc2}
\begin{split}
\lVert \phi_{xx}(\cdot,t)\rVert_{L^2}^2
&= \int_{u_-}^{u_+}\frac{\left(B(y)\Phi_y(y,t)+q(y)^2\Phi_{yy}(y,t)\right)^2}{q(y)}\,dy \\
&\leq C\left(\int_{u_-}^{u_+}q(y)\Phi_y^2(y,t)\,dy+\int_{u_-}^{u_+}q(y)^3\Phi_{yy}^2(y,t)\,dy\right) \\
&\leq C\mathcal{Q}(t).
\end{split}
\end{equation}
Here we used $B=qq_y$, the boundedness of $q_y$, and \eqref{fourth3}.

Collecting \eqref{coerc1} and \eqref{coerc2}, we obtain
\begin{equation}\label{claim}
\mathcal{Q}(t) \geq c\left(\lVert \phi_x(\cdot,t)\rVert_{H^1(\mathbb R)}^2+\lVert (\sqrt{\bar{u}'}\phi)(\cdot,t) \rVert_{L^2(\mathbb{R})}^2 + |\dot{\delta}(t)|^2\right)
\end{equation}
for some $c>0$.

\subsubsection{Estimate of the nonlinear part}
We next estimate the nonlinear term
\begin{equation*}
\mathcal{N}(t) =\int_{\mathbb R} \phi \left( B(\bar{u}+\phi) - B(\bar{u}) - B'(\bar{u})\phi \right)_{xx} \, dx.
\end{equation*}
Integrating by parts, we have
\begin{equation} \label{N_intparts}
\begin{split}
\mathcal N(t) &= -\int_{\mathbb R}\phi_x \left(B(\bar u+\phi)-B(\bar u)-B'(\bar u)\phi\right)_x \, dx \\
& = -\int_{\mathbb R}\bar u'R(\phi)\phi_x \, dx - \int_{\mathbb R}B''(\bar u)\phi\phi_x^2 \, dx - \int_{\mathbb R}R(\phi)\phi_x^2 \, dx,
\end{split}
\end{equation}
where
\begin{equation}\label{Rphi_def}
R(\phi) := B'(\bar u+\phi) - B'(\bar u)-B''(\bar u)\phi.
\end{equation}
Note that, by Taylor's theorem,
\begin{equation*}
R(\phi) = \phi^2\int_0^1(1-s)B'''(\bar u+s\phi) \, ds,
\end{equation*}
which together with \eqref{eq:potential-comparison} yields
\begin{equation} \label{Rphi_bd}
0 \le R(\phi) \le C\phi^2.
\end{equation}
Thus, the last term in the expression \eqref{N_intparts} for $\mathcal{N}$ is non-positive, and hence
\begin{equation}\label{Nt}
\mathcal N(t)\le\left|\int_{\mathbb R}\bar u'R(\phi)\phi_x\,dx\right|+\left|\int_{\mathbb R}B''(\bar u)\phi\phi_x^2\,dx\right|.
\end{equation}
By \eqref{Rphi_bd}, the first term on the right-hand side is bounded as
\begin{equation*} 
\left|\int_{\mathbb R}\bar u'R(\phi)\phi_x \, dx \right| \le C\int_{\mathbb R}\bar u'|\phi|^2|\phi_x| \, dx. 
\end{equation*}
Applying H\"older's inequality and the one-dimensional Sobolev inequality, we have
\begin{equation*} 
\begin{split} 
\int_{\mathbb R}\bar u'|\phi|^2|\phi_x| \, dx &\le \lVert \sqrt{\bar{u}'}\phi \rVert_{L^\infty} \int_\mathbb{R} \sqrt{\bar{u}'} |\phi| |\phi_x| \, dx \\
&\le \|\phi\|_{L^2}\|\sqrt{\bar u'}\phi\|_{L^\infty}\|\sqrt{\bar u'}\phi_x\|_{L^2}\\ 
&\le C\|\phi\|_{L^2}\|\sqrt{\bar u'}\phi\|_{H^1}\|\phi_x\|_{L^2}. 
\end{split} 
\end{equation*}
Moreover, since $|\bar u''|\le C\bar u'$ by \eqref{eq:front-bounds}, we have
\begin{equation*} 
\|\sqrt{\bar u'}\phi\|_{H^1}^2\le C\left(\int_{\mathbb R}\bar u'\phi^2\,dx+\int_{\mathbb R}\bar u'\phi_x^2\,dx\right)\le C\left(\|\sqrt{\bar u'}\phi\|_{L^2}^2+\|\phi_x\|_{L^2}^2\right). 
\end{equation*}
Thus,
\begin{equation}\label{Ntfirst} 
\left|\int_{\mathbb R}\bar u'R(\phi)\phi_x\,dx\right|\le C\|\phi(\cdot,t)\|_{L^2}\left(\|\sqrt{\bar u'}\phi\|_{L^2}^2+\|\phi_x\|_{L^2}^2\right). 
\end{equation}

For the second term, since \(B''(\bar{u})\) is bounded, we obtain
\begin{equation*}
\left|\int_{\mathbb R}B''(\bar{u})\phi\phi_x^2\,dx\right|\leq C\int_{\mathbb R}|\phi|\phi_x^2\,dx.
\end{equation*}
By H\"older's inequality and the one-dimensional Gagliardo--Nirenberg inequality, the right-hand side is estimated as
\begin{equation} \label{Ntsecond}
\begin{split}
\int_{\mathbb R}|\phi|\phi_x^2\,dx \leq \|\phi\|_{L^2}\|\phi_x\|_{L^4}^2  \leq C\|\phi\|_{L^2}\|\phi_x\|_{L^2}^{3/2}\|\phi_{xx}\|_{L^2}^{1/2}  \leq C\|\phi\|_{L^2}\|\phi_x\|_{H^1}^2.
\end{split}
\end{equation}

Since $\lVert \phi(\cdot,t) \rVert_{L^2} \leq \varepsilon_1$, combining \eqref{Nt}--\eqref{Ntsecond}, we obtain
\begin{equation} \label{Ntbound}
\mathcal{N}(t) \leq C \varepsilon_1 \left( \lVert \phi_x \rVert_{H^1}^2 + \lVert (\sqrt{\bar{u}'}\phi) \rVert_{L^2}^2 \right).
\end{equation}

\subsubsection{Completion of the proof}
Using the estimates \eqref{claim} and \eqref{Ntbound} in \eqref{E_identity}, we obtain
\begin{equation*}
\begin{split}
\frac{1}{2}\frac{d}{dt} \int_\mathbb{R} \phi^2 \, dx + c \left( \lVert \phi_x(\cdot,t) \rVert_{H^1}^2 + \lVert (\sqrt{\bar{u}'}\phi)(\cdot,t) \rVert_{L^2}^2 + |\dot{\delta}(t)|^2 \right) \leq C \varepsilon_1 \left( \lVert \phi_x \rVert_{H^1}^2 + \lVert (\sqrt{\bar{u}'}\phi) \rVert_{L^2}^2 \right)
\end{split}
\end{equation*}
for some $c>0$ and $C>0$. Integrating with respect to $t$, we obtain
\begin{equation*}
\begin{split}
\lVert \phi(\cdot,t) \rVert_{L^2}^2 + c \int_0^t \left( \lVert \phi_x(\cdot,s) \rVert_{H^1}^2 + \lVert (\sqrt{\bar{u}'}\phi)(\cdot,s) \rVert_{L^2}^2 + |\dot{\delta}(s)|^2 \right) \, ds \leq \lVert \phi(\cdot,0) \rVert_{L^2}^2
\end{split}
\end{equation*}
for sufficiently small $\varepsilon_1>0$.

\section{\texorpdfstring{From $L^2$-contraction to $L^2$-asymptotic orbital stability}{From L2-contraction to L2-asymptotic orbital stability}} \label{sec3}
We combine the \(L^2\)-contraction estimate from Section~\ref{sec2} with a far-field semigroup argument to prove the \(L^2\)-asymptotic orbital stability of the front \(\bar u\), namely \eqref{L2stability}--\eqref{deltalimit} in Theorem~\ref{Main}.

Recall the perturbation equation with the shift chosen as in \eqref{delta}:
\begin{equation} \label{eq:pert'}
\phi_t = \partial_x^2 (- \partial_x^2\phi + B'(\bar{u})\phi) + \partial_x^2 N(\phi) - \eta \left( \int_\mathbb{R} \bar{u}' \phi \, dx \right) (\bar{u}' + \partial_x \phi),
\end{equation}
where
\begin{equation*}
N(\phi) := B(\bar u+\phi)-B(\bar u)-B'(\bar u)\phi.
\end{equation*}
Using the far-field coefficient $B_0'(u_-)=B_0'(u_+)=\kappa^2L^2$ of the reference quartic potential, we define
\begin{equation}\label{LP_def}
L_\infty\phi:=\partial_x^2(-\partial_x^2\phi+\kappa^2L^2\phi),\qquad \mathcal P\phi:=\bar u'\int_{\mathbb R}\bar u'\phi\,dx,
\end{equation}
and the damped reference far-field operator
\begin{equation}\label{A_def}
\mathcal A_\infty\phi:=L_\infty\phi-\eta\mathcal P\phi.
\end{equation}
Then the perturbation equation \eqref{eq:pert'} can be written as
\begin{equation} \label{eq:far}
\phi_t = \mathcal{A}_\infty \phi + \partial_x \mathcal{F}(\phi),
\end{equation}
where
\begin{equation} \label{Fdef}
\mathcal{F}(\phi) := \big( B'(\bar u)-\kappa^2L^2 \big) \partial_x\phi+B''(\bar{u}) \bar{u}'\phi+\partial_xN(\phi)-\eta\phi\int_{\mathbb R}\bar u'\phi \, dx.
\end{equation}
We begin with the basic semigroup property of $\mathcal A_\infty$.

\begin{lemma} \label{semi}
The operator $\mathcal{A}_\infty$ generates a $C_0$ contraction semigroup $e^{t\mathcal{A}_\infty}$ on $L^2(\mathbb{R})$.
\end{lemma}

\begin{proof}
Let \(f\in H^4(\mathbb R)\). Using the definitions of \(L_\infty\), \(\mathcal P\), and \(\mathcal A_\infty\) in \eqref{LP_def}--\eqref{A_def}, together with integration by parts, we obtain
\begin{equation*}
\langle f,\mathcal A_\infty f\rangle = - \|f_{xx}\|_{L^2}^2-\kappa^2L^2\|f_x\|_{L^2}^2-\eta|\langle f,\bar u'\rangle|^2\le0,
\end{equation*}
where \(\langle\cdot,\cdot\rangle\) denotes the \(L^2\)-inner product. Hence \(\mathcal A_\infty\) is dissipative.

We next show that \(\mathcal A_\infty\) is maximal dissipative. The operator \(L_\infty\) is a real constant-coefficient self-adjoint operator on \(H^4(\mathbb R)\), and \(\mathcal P\) is bounded and self-adjoint on \(L^2(\mathbb R)\). Therefore \(\mathcal A_\infty=L_\infty-\eta\mathcal P\), with domain \(H^4(\mathbb R)\), is self-adjoint. Since it is also dissipative, it is maximal dissipative. By the Lumer--Phillips theorem, \(\mathcal A_\infty\) generates a \(C_0\) contraction semigroup \(e^{t\mathcal A_\infty}\) on \(L^2(\mathbb R)\).
\end{proof}

We apply Duhamel's principle to \eqref{eq:far} and obtain
\begin{equation*}
\phi(\cdot,t) = e^{(t-T)\mathcal A_\infty} \phi(\cdot,T) + \int_T^t e^{(t-s)\mathcal A_\infty} \partial_x \mathcal{F}(\phi(\cdot,s)) \, ds, \quad t \geq T
\end{equation*}
for any \(T>0\), where \(e^{t\mathcal A_\infty}\) denotes the semigroup generated by \(\mathcal A_\infty\) in Lemma~\ref{semi}. This Duhamel representation is understood as the \(L^2\)-valued limit of the corresponding formula for smooth approximating sources; this is justified in Appendix~\ref{app:duhamel-justification}. Taking the \(L^2\) norm and using the triangle inequality, we have
\begin{equation*}
\lVert \phi(\cdot,t) \rVert_{L^2} \leq \lVert e^{(t-T)\mathcal A_\infty} \phi(\cdot,T) \rVert_{L^2} + \left\lVert \int_T^t e^{(t-s)\mathcal A_\infty} \partial_x \mathcal{F}(\phi(\cdot,s)) \, ds \right\rVert_{L^2}.
\end{equation*}
Thus the proof of \eqref{L2stability} can be reduced to showing that
\begin{equation} \label{S_decay}
\lim_{t \to \infty}{\lVert e^{(t-T)\mathcal A_\infty} \phi(\cdot,T) \rVert_{L^2}} = 0 \quad \text{for each fixed } T > 0,
\end{equation}
and
\begin{equation} \label{F_decay}
\lim_{T\to\infty}
\sup_{t\ge T} \left\lVert \int_T^t e^{(t-s)\mathcal A_\infty} \partial_x \mathcal{F}(\phi(\cdot,s)) \, ds \right\rVert_{L^2}=0.
\end{equation}
Once \eqref{L2stability} is established, \eqref{deltalimit} follows directly from the definition of $\dot\delta(t)$:
\begin{equation*}
|\dot\delta(t)|=\eta \left| \int_\mathbb{R} \bar{u}'\phi \, dx \right| \le \eta\| \bar{u}' \|_{L^2}\| \phi(\cdot,t) \|_{L^2} \to 0 \quad \text{as } t\to\infty.
\end{equation*}

The remainder of this section is devoted to proving \eqref{S_decay} and \eqref{F_decay}, thereby completing the proof of both \eqref{L2stability} and \eqref{deltalimit}.

\subsection{\texorpdfstring{Proofs of \eqref{S_decay} and \eqref{F_decay}}{Proofs of (4.6) and (4.7)}}

\subsubsection{Proof of \eqref{S_decay}}
The following lemma implies \eqref{S_decay}.

\begin{lemma} \label{A_decay}
Let \(e^{t\mathcal A_\infty}\) be the contraction semigroup generated by \(\mathcal A_\infty\) in Lemma~\ref{semi}. Then, for any $f\in L^2(\mathbb{R})$, we have
\begin{equation*}
\lim_{t\to\infty}\|e^{t\mathcal A_\infty}f\|_{L^2(\mathbb{R})}=0.
\end{equation*}
\end{lemma}

\begin{proof}
By the Arendt--Batty--Lyubich--V\~u theorem for strong stability of bounded \(C_0\) semigroups \cite{AB, LP}, it suffices to verify that \(\sigma(\mathcal A_\infty)\cap i\mathbb R\) is countable and that \(\sigma_p(\mathcal A_\infty^*)\cap i\mathbb R=\emptyset\). Here \(\mathcal A_\infty^*\) denotes the adjoint of \(\mathcal A_\infty\), while \(\sigma(\cdot)\) and \(\sigma_p(\cdot)\) denote the spectrum and the point spectrum, respectively.

From the proof of Lemma~\ref{semi}, \(\mathcal A_\infty\) is self-adjoint and non-positive on \(L^2(\mathbb R)\), with domain \(H^4(\mathbb R)\). Hence its spectrum is contained in the non-positive real axis:
\begin{equation*}
\sigma(\mathcal A_\infty)\subset\{z\in\mathbb C:\operatorname{Im}z=0,\ \operatorname{Re}z\le0\}.
\end{equation*}
In particular, \(\sigma(\mathcal A_\infty)\cap i\mathbb R\subset\{0\}\), so \(\sigma(\mathcal A_\infty)\cap i\mathbb R\) is countable.

It remains to rule out eigenvalues on the imaginary axis. Since \(\mathcal A_\infty\) is self-adjoint, it is enough to show that \(0\) is not an eigenvalue. Suppose, to the contrary, that \(0\) is an eigenvalue with associated eigenfunction \(g\in H^4(\mathbb R)\). Then
\begin{equation*}
0=-\langle g,\mathcal A_\infty g\rangle =\|g_{xx}\|_{L^2}^2+\kappa^2L^2\|g_x\|_{L^2}^2+\eta|\langle g,\bar u'\rangle|^2.
\end{equation*}
Thus \(g_x=0\), and since \(g\in L^2(\mathbb R)\), we have \(g=0\). This contradicts the assumption that \(g\) is an eigenfunction.
\end{proof}

\subsubsection{Proof of \eqref{F_decay}} \label{sec3.1.2}
The calculation below is justified for \(L^2\)-sources by the approximation argument in Appendix~\ref{app:duhamel-justification}. 

Fix $T>0$ and set
\begin{equation} \label{wdef}
w(t):=\int_T^t e^{(t-s)\mathcal A_\infty} \partial_x \mathcal{F}(\phi(\cdot,s))\,ds.
\end{equation}
Then $w$ solves
\begin{equation*}
w_t= \mathcal{A}_\infty w+\partial_x \mathcal{F}(\phi), \quad w(T)=0.
\end{equation*}
Taking the $L^2$-inner product with $w$, we obtain
\begin{equation*}
\frac12\frac{d}{dt}\|w\|_{L^2}^2+Q_\infty[w]=\langle w,\partial_x \mathcal{F}(\phi)\rangle=-\langle w_x,\mathcal{F}(\phi)\rangle,
\end{equation*}
where we set
\begin{equation*}
Q_\infty[w]:=-\langle w, \mathcal{A}_\infty w\rangle = \lVert w_{xx} \rVert_{L^2}^2 + \kappa^2 L^2 \lVert w_x \rVert_{L^2}^2 + \eta \left| \langle w,\bar{u}' \rangle \right|^2 \geq 0.
\end{equation*}
Since $\kappa^2L^2\|w_x\|_{L^2}^2\le Q_\infty[w]$, we have
\begin{equation*}
\frac12\frac{d}{dt}\|w\|_{L^2}^2+Q_\infty[w]\le \frac{1}{\kappa L}Q_\infty[w]^{1/2}\|\mathcal{F}(\phi(\cdot,t))\|_{L^2}\le \frac12Q_\infty[w]+\frac{1}{2\kappa^2L^2}\|\mathcal{F}(\phi(\cdot,t))\|_{L^2}^2.
\end{equation*}
Thus
\begin{equation*}
\frac{d}{dt}\|w\|_{L^2}^2\le \frac{1}{\kappa^2L^2}\|\mathcal{F}(\phi(\cdot,t))\|_{L^2}^2.
\end{equation*}
Integrating from $T$ to $t$ yields
\begin{equation*}
\|w(t)\|_{L^2}^2 \leq \frac{1}{\kappa^2L^2}\int_T^t \|\mathcal{F}(\phi(\cdot,s))\|_{L^2}^2 \,ds \leq \frac{1}{\kappa^2L^2}\|\mathcal{F}(\phi)\|_{L^2([T,\infty);L^2)}^2.
\end{equation*}
Taking the supremum over $t\geq T$ gives
\begin{equation} \label{FxF}
\sup_{t\geq T}\|w(t)\|_{L^2}\leq \frac{1}{\kappa L}\|\mathcal{F}(\phi)\|_{L^2([T,\infty);L^2)}.
\end{equation}
The lemma below provides an estimate for the right-hand side of \eqref{FxF}.

\begin{lemma}\label{lem:source-tail-asymp}
Assume the hypotheses of Theorem~\ref{Main}. Then $\mathcal{F}(\phi)$ defined by \eqref{Fdef} satisfies
\begin{equation}\label{eq:F-tail-asymp}
\|\mathcal{F}(\phi)\|_{L^2([T,\infty);L^2)}^2\le C\int_T^\infty \mathcal{D}(t)\,dt
\end{equation}
for any $T>0$, where $C>0$ is independent of $T$, and $\mathcal{D}(\cdot)$ is defined as in Proposition~\ref{apriori}.
\end{lemma}

\begin{proof}
It suffices to prove the pointwise estimate
\begin{equation*}
\|\mathcal{F}(\phi)(\cdot,t)\|_{L^2}^2\le C \mathcal{D}(t)
\end{equation*}
for a.e. \(t>0\). Fix a time \(t>0\) such that \(\mathcal D(t)<\infty\), and suppress the time dependence in the notation. Taking the $L^2$ norm in $x$ on both sides of \eqref{Fdef}, we have
\begin{equation} \label{FL2}
\lVert\mathcal F(\phi)\rVert_{L^2} \le \lVert \big( B'(\bar{u}) - \kappa^2L^2 \big) \phi_x \rVert_{L^2} + \lVert B''(\bar{u})\bar{u}'\phi \rVert_{L^2} + \lVert\partial_xN(\phi)\rVert_{L^2} + \left\lVert\eta\phi\int_{\mathbb R}\bar u'\phi \, dx \right\rVert_{L^2}.
\end{equation}
We estimate each term on the right-hand side. By \eqref{eq:B}, we have
\begin{equation*}
B'(\bar{u}) - \kappa^2L^2 = B_0'(\bar u)+h''(\bar u)-\kappa^2L^2=-6\kappa^2(\bar u-u_-)(u_+-\bar u)+h''(\bar u),
\end{equation*}
and
\begin{equation*}
B''(\bar u)=B_0''(\bar u)+h'''(\bar u)=6\kappa^2(2\bar u-u_--u_+)+h'''(\bar u).
\end{equation*}
Since $\|h''\|_{L^\infty(\mathbb R)}+\|h'''\|_{L^\infty(\mathbb R)}\le\rho_F$ by \eqref{ppsize}, these identities yield
\begin{equation*}
|B'(\bar{u}) - \kappa^2L^2| \le \frac32\kappa^2L^2+\rho_F,\quad |B''(\bar u)|\le6\kappa^2L+\rho_F.
\end{equation*}
Using these bounds and $0<\bar u'\le C$ from Lemma~\ref{lem:front-properties}, we obtain
\begin{equation} \label{eq:linear-source-bound-asymp}
\begin{split}
\lVert \big( B'(\bar{u}) - \kappa^2L^2 \big) \phi_x \rVert_{L^2} &\leq C \lVert \phi_x \rVert_{L^2}, \quad \lVert B''(\bar{u})\bar{u}'\phi \rVert_{L^2} \leq C \lVert \sqrt{\bar{u}'}\phi \rVert_{L^2}.
\end{split}
\end{equation}
We next estimate $\partial_xN(\phi)$. Differentiating $N(\phi)$ and using \eqref{Rphi_def}, we obtain
\begin{equation*} 
\partial_xN(\phi)=\bar u'R(\phi)+B''(\bar u)\phi\phi_x+R(\phi)\phi_x. 
\end{equation*}
Taking the $L^2$ norm in $x$ and squaring, and using \eqref{Rphi_bd} and the boundedness of $B''(\bar u)$, we obtain
\begin{equation} \label{basicNineq}
\begin{split}
\|\partial_xN(\phi)\|_{L^2}^2 & \le C \left( \|\bar u'\phi^2\|_{L^2}^2 + \|\phi\phi_x\|_{L^2}^2 + \|\phi^2\phi_x\|_{L^2}^2 \right) \\
& \leq C\left(\|\phi\|_{L^\infty}^2\|\sqrt{\bar u'}\,\phi\|_{L^2}^2+\|\phi\|_{L^\infty}^2\|\phi_x\|_{L^2}^2+\|\phi\|_{L^\infty}^4\|\phi_x\|_{L^2}^2\right).
\end{split}
\end{equation}
By the $L^2$-contraction estimate \eqref{cont}, we have
\begin{equation*}
\|\phi(\cdot,t)\|_{L^2}\le C\varepsilon_0 \quad \text{for all } t \geq0
\end{equation*}
for sufficiently small \(\varepsilon_0>0\). Hence the one-dimensional Gagliardo--Nirenberg inequality gives
\begin{equation}\label{eq:GN-Linfty-asymp}
\|\phi\|_{L^\infty}^2\le C\|\phi\|_{L^2}\|\phi_x\|_{L^2}\le C\varepsilon_0\|\phi_x\|_{L^2}.
\end{equation}
Moreover, integration by parts yields
\begin{equation}\label{eq:x-by-xx-asymp}
\|\phi_x\|_{L^2}^2=-\int_{\mathbb R}\phi\phi_{xx}\,dx\le \|\phi\|_{L^2}\|\phi_{xx}\|_{L^2}\le C\varepsilon_0\|\phi_{xx}\|_{L^2}.
\end{equation}
Also,
\begin{equation}\label{eq:localized-L2-by-Linfty-asymp}
\|\sqrt{\bar u'}\,\phi\|_{L^2}^2\le \|\bar u'\|_{L^1}\|\phi\|_{L^\infty}^2\le C\varepsilon_0\|\phi_x\|_{L^2}.
\end{equation}
Combining \eqref{eq:GN-Linfty-asymp}--\eqref{eq:localized-L2-by-Linfty-asymp} with \eqref{basicNineq}, we obtain
\begin{equation}\label{eq:N-bound-asymp}
\begin{split}
\|\partial_xN(\phi)\|_{L^2}^2 &\leq C \varepsilon_0^2 \left( \lVert \phi_x \rVert_{L^2}^2 + \lVert \phi_x \rVert_{L^2}\lVert \phi_{xx} \rVert_{L^2} + \lVert \phi_x \rVert_{L^2}^4 \right) \\
& \leq C \varepsilon_0^2 \left( \lVert \phi_x \rVert_{L^2}^2 + \lVert \phi_{xx} \rVert_{L^2}^2 \right)
\end{split}
\end{equation}
for sufficiently small $\varepsilon_0>0$. For the last term of \eqref{FL2}, we use \eqref{cont} and \eqref{delta} to obtain
\begin{equation}\label{eq:phase-source-bound-asymp}
\left\|\eta\phi\int_{\mathbb R}\bar u'\phi\,dx\right\|_{L^2}^2=|\dot\delta(t)|^2\|\phi\|_{L^2}^2\le C\varepsilon_0^2|\dot\delta(t)|^2.
\end{equation}

Combining \eqref{FL2}, \eqref{eq:linear-source-bound-asymp}, \eqref{eq:N-bound-asymp}, and \eqref{eq:phase-source-bound-asymp}, we conclude that
\begin{equation*}
\|\mathcal{F}(\phi)(\cdot,t)\|_{L^2}^2\le C(\|\phi_x(\cdot,t)\|_{L^2}^2+\|\phi_{xx}(\cdot,t)\|_{L^2}^2+\|\sqrt{\bar u'}\,\phi(\cdot,t)\|_{L^2}^2+|\dot\delta(t)|^2)\leq C\mathcal{D}(t).
\end{equation*}
Integrating this estimate over $[T,\infty)$ gives \eqref{eq:F-tail-asymp}.
\end{proof}

By \eqref{wdef}, \eqref{FxF}, and Lemma~\ref{lem:source-tail-asymp}, for any $T>0$, we have
\begin{equation*}
\sup_{t\ge T}\left\|\int_T^t e^{(t-s)\mathcal A_\infty}\partial_x \mathcal{F}(\phi(\cdot,s))\,ds\right\|_{L^2}
\le \frac{1}{\kappa L}\|\mathcal{F}(\phi)\|_{L^2([T,\infty);L^2)}
\le C\left(\int_T^\infty \mathcal{D}(s)\,ds\right)^{1/2}.
\end{equation*}
Since \(\mathcal{D}\in L^1([0,\infty))\) by \eqref{aprioriest}, we have
\begin{equation*}
\int_T^\infty \mathcal{D}(s)\,ds\to0\quad\text{as }T\to\infty.
\end{equation*}
Hence \eqref{F_decay} follows. Together with \eqref{S_decay}, this proves \eqref{L2stability}.

\section{Proof of the Poincar\'e-type inequalities} \label{sec4}
We prove Proposition~\ref{fourth} and Corollary~\ref{fourth-cor}, which were stated in Section~\ref{sec2.2}. Recall that
\begin{equation*}
q(y) = \sqrt{2F(y)}, \quad V(y) = - \frac{B''(y)B(y)+B'''(y)q(y)^2}{2q(y)}, \quad y \in (u_-,u_+).
\end{equation*}
For $\Phi$ as in Proposition~\ref{fourth}, we define
\begin{align}
M[\Phi] & := \int_{u_-}^{u_+}\Phi(y)\,dy, \label{eq:general-M} \\
J_F[\Phi] & := \int_{u_-}^{u_+}q(y)\Phi_y(y)^2\,dy, \label{eq:general-J} \\
Q_F[\Phi] & := \int_{u_-}^{u_+}q(y)^3\Phi_{yy}(y)^2\,dy + \int_{u_-}^{u_+}V(y)\Phi(y)^2\,dy, \label{eq:general-Q} \\
N[\Phi] & := \int_{u_-}^{u_+}\Phi^2 \, dy. \label{eq:general-N}
\end{align}
With this notation, Proposition~\ref{fourth} can be restated as follows.

\begin{propositionprime} 
There exist constants $\rho_0>0$ and $c_{\kappa,L}>0$ such that, if $F$ satisfies \eqref{p-compat} and $\rho_F\le\rho_0$, then, for any $\Phi\in L^2(u_-,u_+)$ satisfying
\begin{equation*}
q^{1/2}\Phi_y \in L^2(u_-,u_+), \quad q^{3/2} \Phi_{yy} \in L^2(u_-,u_+),
\end{equation*}
we have
\begin{equation} \label{eq:generalized-fourth-order-Poincare}
Q_F[\Phi] + \eta M[\Phi]^2 \geq c_{\kappa,L} \left( J_F[\Phi] + M[\Phi]^2 \right),
\end{equation}
where $\eta=4\kappa^3L$. Moreover, there exists a constant $C_{\kappa,L}>0$ such that
\begin{equation} \label{eq:generalized-Poincare}
N[\Phi] \leq C_{\kappa,L} \left( J_F[\Phi] + M[\Phi]^2 \right).
\end{equation}
\end{propositionprime}
We first prove the inequalities for the exact quartic potential $F_0$. In Sections~\ref{sec4.1}--\ref{sec:quartic-completion}, we consider $F=F_0$ and reduce the proof to normalized inequalities on the reference interval $[-1,1]$. After establishing these inequalities, we relate the general functionals \eqref{eq:general-M}--\eqref{eq:general-N} to their normalized counterparts through scaling, which yields \eqref{eq:generalized-fourth-order-Poincare} and \eqref{eq:generalized-Poincare} for $F=F_0$. Finally, we extend these inequalities to perturbed potentials in Section~\ref{sec:potential-perturbation}.

\subsection{Poincar\'e-type inequalities for the normalized functionals} \label{sec4.1}
We introduce the normalized functionals \(M_0\), \(J_0\), \(Q_0\), and \(N_0\) on \([-1,1]\) by
\begin{equation}\label{eq:normalized-forms}
\begin{aligned}
M_0[\Psi] &:= \int_{-1}^{1}\Psi(z)\,dz,\\
J_0[\Psi] &:= \int_{-1}^{1}(1-z^2)\Psi_z(z)^2\,dz,\\
Q_0[\Psi] &:= \frac18\int_{-1}^{1}(1-z^2)^3 \Psi_{zz}(z)^2\,dz +\frac34\int_{-1}^{1}(3z^2-1)\Psi(z)^2\,dz, \\
N_0[\Psi] & := \int_{-1}^{1}\Psi^2(z) \, dz.
\end{aligned}
\end{equation}

We prove the normalized coercivity estimate.

\begin{lemma}[A second-order weighted Poincar\'e inequality] \label{CH-Poincare}
For any \(\Psi \in L^2(-1,1)\) such that
\begin{equation*}
(1-z^2)^{1/2}\Psi_z\in L^2(-1,1), \quad (1-z^2)^{3/2}\Psi_{zz} \in L^2(-1,1),
\end{equation*}
we have
\begin{equation}\label{eq:CH-Poincare-main}
Q_0[\Psi]+M_0[\Psi]^2 \ge \frac{3}{56}\left(J_0[\Psi]+M_0[\Psi]^2\right).
\end{equation}
\end{lemma}

\begin{proof}
Let $X$ be the weighted space endowed with the norm
\begin{equation*}
\|\Psi\|_X^2:=\int_{-1}^1\Psi^2\,dz+\int_{-1}^1(1-z^2)\Psi_z^2\,dz+\int_{-1}^1(1-z^2)^3\Psi_{zz}^2\,dz.
\end{equation*}
Then polynomials are dense in $X$ with respect to $\|\cdot\|_X$; see Appendix~\ref{density}. Since $M_0$, $J_0$, and $Q_0$ are continuous with respect to this norm, it suffices to prove \eqref{eq:CH-Poincare-main} for polynomials.

For $n\ge 0$, let $P_n$ denote the Legendre polynomial normalized by $P_n(1)=1$, so that
\begin{equation} \label{Le1}
-\big((1-z^2)P_n'\big)'=n(n+1)P_n
\end{equation}
and
\begin{equation} \label{Le2}
\int_{-1}^1 P_nP_m\,dz=\frac{2}{2n+1}\delta_{nm},
\end{equation} 
where $\delta_{nm}$ is the Kronecker delta.
This follows from the fact that the Legendre polynomials form the orthogonal eigenbasis of the self-adjoint Sturm--Liouville problem \eqref{Le1}. 

For a polynomial $\Psi$, we write its finite Legendre expansion as
\begin{equation*}
\Psi(z) = \sum_{n\geq 0} a_n P_n(z), \quad a_n = \frac{2n+1}{2} \int_{-1}^1 \Psi(z)P_n(z) \, dz.
\end{equation*}
Multiplying \eqref{Le1} by $P_m$ and integrating the resulting identity from $z=-1$ to $z=1$, we observe that
\begin{equation*}
\begin{split}
\int_{-1}^1 - \big( (1-z^2)P_n'(z) \big)' P_m(z) \, dz 
& = \int_{-1}^1 (1-z^2)P_n'(z)P_m'(z) \, dz \\
&  = n(n+1) \int_{-1}^1 P_n(z)P_m(z) \, dz = \frac{2n(n+1)}{2n+1} \delta_{mn},
 \end{split}
\end{equation*}
where in the last equality we used \eqref{Le2}. Using this, we have
\begin{equation} \label{eq:J-Legendre}
\begin{split}
J_0[\Psi] &= \sum_{n \geq 1} \sum_{m \geq 1} a_n a_m \int_{-1}^1 (1-z^2) P_n'(z) P_m'(z) \, dz \\
&= \sum_{n\geq 1} \sum_{m \geq 1} a_n a_m \frac{2n(n+1)}{2n+1}\delta_{nm} \\
&= \sum_{n \geq 1} g_n a_n^2, \qquad g_n := \frac{2n(n+1)}{2n+1}.
\end{split}
\end{equation}
Moreover, by the coefficient calculation in Appendix~\ref{app:Q0-coefficients}, we have
\begin{equation}\label{eq:Q0-Legendre}
Q_0[\Psi]=\sum_{n\ge0}d_na_n^2+2\sum_{n\ge0}e_na_na_{n+2},
\end{equation}
where
\begin{equation} \label{dnen}
\begin{split}
d_n &=\frac{n^2(n+1)^2(n^2+n+1)}{2(2n-1)(2n+1)(2n+3)}, \\
e_n &=-\frac{(n+1)(n+2)(n^4+6n^3+5n^2-12n-18)}{4(2n+1)(2n+3)(2n+5)}.
\end{split}
\end{equation}
In particular,
\begin{equation*}
d_0 = 0, \quad e_0=\frac35, \quad e_1=\frac9{35}, \quad e_n<0\quad\text{for }n\ge2.
\end{equation*}

We first prove the mean-zero case, where $M_0[\Psi]=0$. By \eqref{Le2}, it holds that
\begin{equation*}
0=M_0[\Psi] = \sum_{n\geq 0} a_n \int_{-1}^1 P_n(z) \, dz = \sum_{n\geq 0} a_n \int_{-1}^1 P_n(z)P_0(z) \, dz = \sum_{n\geq 0} a_n \frac{2}{2n+1}\delta_{n0} =2a_0.
\end{equation*}
Using
\begin{equation*}
2e_na_na_{n+2} \ge -|e_n|a_n^2-|e_n|a_{n+2}^2,
\end{equation*}
and noting that the \(2e_0a_0a_2\) term vanishes because \(a_0=0\), we obtain
\begin{equation*}
Q_0[\Psi] = \sum_{n\geq 0}  d_n a_n^2 + 2 \sum_{n\geq 0} e_n a_n a_{n+2} \geq \sum_{n\geq 1} m_na_n^2,
\end{equation*}
where
\begin{equation*}
\begin{split}
m_1 &=d_1-|e_1|, \quad m_2=d_2-|e_2|, \\
m_n &=d_n-|e_{n-2}|-|e_n|, \quad n\ge3.
\end{split}
\end{equation*}
A direct substitution gives
\begin{equation*}
m_1=\frac17=\frac{3}{28}g_1,\quad m_2=\frac45\ge \frac{3}{28}g_2,\quad m_3-\frac{3}{28}g_3=\frac{355}{539}>0,
\end{equation*}
and, for $n\ge4$,
\begin{equation*}
m_n-\frac{3}{28}g_n=\frac{3 \left(40n^6 + 120n^5 + 52n^4- 96n^3 - 281n^2 - 213n + 126\right)}{14(2n-3)(2n-1)(2n+1)(2n+3)(2n+5)}.
\end{equation*}
The denominator is positive for $n\ge4$, and the numerator is positive since
\begin{equation} \label{denu}
\begin{split}
& 40n^6+120n^5+52n^4-96n^3-281n^2-213n+126 \\
& \quad \geq 40n^6-96n^3-281n^2-213n+126 \\
& \quad \ge 40n^6-590n^4+126>0,\quad n\ge4.
\end{split}
\end{equation}
In \eqref{denu}, the first inequality follows by dropping the nonnegative terms $120n^5+52n^4$, the second from $96n^3+281n^2+213n\le590n^4$, and the last from $40n^6-590n^4+126=n^4(40n^2-590)+126>0$. Hence $m_n\ge \frac{3}{28}g_n$ for all $n\ge1$, and consequently
\begin{equation}
\label{eq:mean-zero-Poincare}
Q_0[\Psi] \ge \frac{3}{28} \sum_{n\geq 1} g_na_n^2 = \frac{3}{28} J_0[\Psi].
\end{equation}

We now treat the general case. Let $\widetilde{\Psi} := \Psi -a_0$. Then, since \(M_0[\Psi]=2a_0\), 
\begin{equation*}
M_0[\widetilde{\Psi}] = \int_{-1}^1 \widetilde{\Psi} \, dz = \int_{-1}^1 (\Psi - a_0) \, dz = M_0[\Psi]-2a_0=0. 
\end{equation*}
Moreover,
\begin{equation*}
J_0[\widetilde{\Psi}]=\int_{-1}^1(1-z^2)\widetilde{\Psi}_z^2\,dz=\int_{-1}^1(1-z^2)\Psi_z^2\,dz = J_0[\Psi]
\end{equation*}
and, by \eqref{eq:Q0-Legendre},
\begin{equation*}
Q_0[\Psi] = Q_0[\widetilde{\Psi}] + 2e_0 a_0 a_2 = Q_0[\widetilde{\Psi}] + \frac{6}{5} a_0a_2.
\end{equation*}
By \eqref{eq:J-Legendre},
\begin{equation*}
a_2^2 \le \frac{1}{g_2}J_0[\Psi] = \frac{5}{12}J_0[\Psi].
\end{equation*}
Hence Young's inequality gives
\begin{equation*}
\left|\frac65a_0a_2\right| \leq \frac{9}{70} a_2^2 + \frac{1}{4 \cdot \frac{9}{70}} \left( \frac{6}{5}a_0 \right)^2 \le \frac{3}{56}J_0[\Psi] + \frac{14}{5}a_0^2.
\end{equation*}
Using the estimate \eqref{eq:mean-zero-Poincare} and $M_0[\Psi]=2a_0$, we obtain
\begin{equation*}
\begin{split}
Q_0[\Psi] +  M_0[\Psi]^2 & = Q_0[\widetilde{\Psi}] + \frac{6}{5}a_0a_2 +  M_0[\Psi]^2\\
& \geq \frac{3}{28}J_0[\Psi] + \frac{6}{5}a_0a_2 +  M_0[\Psi]^2 \\
& \geq \frac{3}{28}J_0[\Psi] - \frac{3}{56} J_0[\Psi] - \frac{14}{5} a_0^2 +  M_0[\Psi]^2 \\
& = \frac{3}{56} J_0[\Psi] + \frac{3}{10} M_0[\Psi]^2.
\end{split}
\end{equation*}
Since $\textstyle \frac{3}{10} > \frac{3}{56}$, this implies \eqref{eq:CH-Poincare-main}.
\end{proof}

\begin{remark} \label{notsharp}
The constant $\frac{3}{56}$ in \eqref{eq:CH-Poincare-main} is not claimed to be sharp; it is only a convenient explicit positive coercivity constant. A natural problem is to determine
the optimal constant
\begin{equation*}
    c_\ast
    :=
    \inf_{\Psi\in X_0,\ \Psi\neq0}
    \frac{Q_0[\Psi]+M_0[\Psi]^2}
         {J_0[\Psi]+M_0[\Psi]^2},
\end{equation*}
where the infimum is taken over the admissible class
\[
X_0:=\left\{\Psi\in L^2(-1,1):
(1-z^2)^{1/2}\Psi_z\in L^2(-1,1),\ 
(1-z^2)^{3/2}\Psi_{zz}\in L^2(-1,1)\right\}.
\]
By Lemma~\ref{CH-Poincare}, one has \(c_\ast\geq \frac{3}{56}\).

Heuristically, the estimate \eqref{eq:CH-Poincare-main} says that after controlling the constant mode through $M_0[\Psi]^2$, the second-order weighted energy $Q_0$ controls the first-order weighted energy $J_0$. This is nontrivial because the zeroth-order part of $Q_0$ contains the sign-changing weight $3z^2-1$.
\end{remark}

The weighted Poincar\'e estimate corresponding to \eqref{eq:generalized-Poincare} is the following.

\begin{lemma}[\cite{KV1}, Lemma~2.9] \label{Poincare}
For any $\Psi \in L^2(-1,1)$ satisfying $(1-z^2)^{1/2}\Psi_z\in L^2(-1,1)$, it holds that
\begin{equation*}
N_0[\Psi] \leq \frac{1}{2} \left( J_0[\Psi] + M_0[\Psi]^2 \right).
\end{equation*}
\end{lemma}

\begin{proof}
By a standard density argument, it suffices to prove the estimate for polynomials. Let $P_n$ be the Legendre polynomial normalized by $P_n(1)=1$. For a polynomial $\Psi$, we write
\begin{equation*}
\Psi(z)=\sum_{n\geq0}a_nP_n(z).
\end{equation*}
Then, as in the proof of Lemma~\ref{CH-Poincare}, we have
\begin{equation*}
\begin{split}
M_0[\Psi] &= \int_{-1}^1\Psi(z)\,dz = 2a_0, \\
N_0[\Psi] &= \int_{-1}^1\Psi^2(z)\,dz = 2a_0^2 + \sum_{n\geq1} \frac{2}{2n+1}a_n^2, \\
J_0[\Psi] &= \int_{-1}^1(1-z^2)\Psi_z^2\,dz = \sum_{n\geq1} \frac{2n(n+1)}{2n+1}a_n^2.
\end{split}
\end{equation*}
Hence we obtain
\begin{equation*}
\frac12\left(J_0[\Psi]+M_0[\Psi]^2\right) = 2a_0^2 + \sum_{n\geq1} \frac{n(n+1)}{2n+1}a_n^2 \geq 2a_0^2 + \sum_{n\geq1} \frac{2}{2n+1}a_n^2 = N_0[\Psi].
\end{equation*}
The general case follows by density.
\end{proof}

For another proof of Lemma~\ref{Poincare}, we refer the reader to Appendix~B of \cite{KV1}.

\subsection{Scaling relation with the general quadratic-weight case}
We next relate the general quadratic-weight functionals \eqref{eq:general-M}--\eqref{eq:general-N} to the normalized functionals defined in \eqref{eq:normalized-forms}. Set
\begin{equation*}
y_c:=\frac{u_-+u_+}{2},
\end{equation*}
and introduce
\begin{equation*}
z:=\frac{2(y-y_c)}{L} = \frac{2y-(u_-+u_+)}{u_+-u_-}, \quad -1\le z\le 1.
\end{equation*}
Then we have
\begin{equation*}
y=y_c+\frac{L}{2}z, \quad dy=\frac{L}{2}\,dz.
\end{equation*}
For a function \(\Phi\) on \([u_-,u_+]\), we define
\begin{equation}\label{eq:hat-Phi}
\Psi(z) := \Phi\left(y_c+\frac{L}{2}z\right).
\end{equation}
Then, for $F=F_0$, $q$ and $V$ in \eqref{qV} can be rewritten in terms of $z$ as
\begin{equation}\label{eq:q-z-form}
q(y) = \frac{\kappa L^2}{4}(1-z^2), \quad V(y) = \frac32\kappa^3L^2(3z^2-1).
\end{equation}
A direct calculation gives the relations:
\begin{equation}\label{eq:scaling-MJQ}
M[\Phi] = \frac{L}{2}M_0[\Psi], \quad J_{F_0}[\Phi] = \frac{\kappa L}{2}J_0[\Psi], \quad Q_{F_0}[\Phi] = \kappa^3L^3Q_0[\Psi], \quad N[\Phi] = \frac{L}{2} N_0[\Psi],
\end{equation}
where we used
\begin{equation*}
B''(y)=6\kappa^2Lz, \quad B'''(y)=12\kappa^2, \quad \Phi_y=\frac{2}{L}\Psi_z, \quad \Phi_{yy}=\frac{4}{L^2}\Psi_{zz}.
\end{equation*}

\subsection{Proofs for the exact quartic potential} \label{sec:quartic-completion}

We now complete the proofs of Proposition~\ref{fourth} and Corollary~\ref{fourth-cor} for $F=F_0$.

\begin{proof}[Proof of Proposition~\ref{fourth} for $F=F_0$]
Let \(\Phi\) be as in Proposition~\ref{fourth}, and let \(\Psi\) be defined by \eqref{eq:hat-Phi}. By Lemma~\ref{CH-Poincare}, it holds that
\begin{equation}\label{eq:normalized-CH-Poincare-used}
Q_0[\Psi] + M_0[\Psi]^2 \ge \frac{3}{56} \left( J_0[\Psi] + M_0[\Psi]^2 \right).
\end{equation}
Using \eqref{eq:scaling-MJQ} and $\eta = 4\kappa^3L$, we compute
\begin{equation*}
\begin{aligned}
Q_{F_0}[\Phi]+\eta M[\Phi]^2 &= \kappa^3L^3Q_0[\Psi] + \eta\frac{L^2}{4}M_0[\Psi]^2\\
&= \kappa^3L^3 \left( Q_0[\Psi] + \frac{\eta}{4\kappa^3L} M_0[\Psi]^2 \right)\\
&= \kappa^3L^3 \left( Q_0[\Psi] + M_0[\Psi]^2 \right).
\end{aligned}
\end{equation*}
Therefore, using \eqref{eq:normalized-CH-Poincare-used} and \eqref{eq:scaling-MJQ}, we obtain
\begin{equation*}
Q_{F_0}[\Phi]+\eta M[\Phi]^2 \ge \frac{3}{56} \kappa^3L^3  \left( J_0[\Psi] + M_0[\Psi]^2 \right) = \frac{3}{56} \kappa^3L^3 \left( \frac{2}{\kappa L}J_{F_0}[\Phi] + \frac{4}{L^2}M[\Phi]^2 \right).
\end{equation*}
Hence \eqref{eq:generalized-fourth-order-Poincare} holds with, for example,
\begin{equation*}
c_{\kappa,L} := \frac{3}{56} \kappa^3L^3  \min\left\{ \frac{2}{\kappa L}, \frac{4}{L^2} \right\}.
\end{equation*}

We next prove \eqref{eq:generalized-Poincare}. Recall from \eqref{eq:scaling-MJQ} that
\begin{equation*}
N[\Phi] = \frac{L}{2} N_0[\Psi].
\end{equation*}
Applying Lemma~\ref{Poincare} to \(\Psi\), we obtain
\begin{equation*}
\begin{aligned}
N[\Phi] &\le \frac{L}{4} \left( J_0[\Psi] + M_0[\Psi]^2 \right).
\end{aligned}
\end{equation*}
Then, using the relations in \eqref{eq:scaling-MJQ}, we have
\begin{equation*}
N[\Phi] \le \frac{1}{2\kappa}J_{F_0}[\Phi] + \frac{1}{L}M[\Phi]^2.
\end{equation*}
This completes the proof.
\end{proof}

\begin{proof}[Proof of Corollary~\ref{fourth-cor} for $F=F_0$]
Set
\begin{equation*}
\begin{split}
A_{F_0}[\Phi] & := \int_{u_-}^{u_+}q(y)^3\Phi_{yy}(y)^2\,dy, \\
D_{F_0}[\Phi] & := Q_{F_0}[\Phi]+\eta M[\Phi]^2.
\end{split}
\end{equation*}
Then the desired inequality \eqref{fourth3} can be written as
\begin{equation} \label{desired}
J_{F_0}[\Phi] + A_{F_0}[\Phi] \le \widetilde C_{\kappa,L} D_{F_0}[\Phi].
\end{equation}
We prove below that such a constant $\widetilde C_{\kappa,L}>0$ exists.

First, the inequality \eqref{fourth1} in Proposition~\ref{fourth} gives directly
\begin{equation} \label{J_bound}
J_{F_0}[\Phi] + M[\Phi]^2 \le \frac{1}{c_{\kappa,L}} D_{F_0}[\Phi].
\end{equation}
Therefore, it suffices to bound $A_{F_0}[\Phi]$ by $D_{F_0}[\Phi]$. By the explicit formula \eqref{eq:q-z-form}, $V$ extends continuously to $[u_-,u_+]$, so it is bounded from below. Hence there exists a constant $C_V>0$, depending only on $\kappa$ and $L$, such that
\begin{equation*}
V(y)\ge -C_V
\end{equation*}
for all $y\in[u_-,u_+]$. Therefore,
\begin{equation*}
-\int_{u_-}^{u_+}V(y)\Phi(y)^2\,dy\le C_VN[\Phi].
\end{equation*}
Using \eqref{eq:general-Q} with $F=F_0$, we have
\begin{equation*}
A_{F_0}[\Phi] = Q_{F_0}[\Phi]-\int_{u_-}^{u_+}V(y)\Phi(y)^2 \, dy.
\end{equation*}
Since $D_{F_0}[\Phi] = Q_{F_0}[\Phi]+\eta M[\Phi]^2$ and $\eta M[\Phi]^2\ge0$, it follows that
\begin{equation*}
A_{F_0}[\Phi] \le D_{F_0}[\Phi] + C_V N[\Phi].
\end{equation*}
Here, by \eqref{fourth2} in Proposition~\ref{fourth}, there exists a constant $C_{\kappa,L}^{P}>0$ such that
\begin{equation*}
N[\Phi]\le C_{\kappa,L}^{P}\left(J_{F_0}[\Phi]+M[\Phi]^2\right) \le \frac{C_{\kappa,L}^{P}}{c_{\kappa,L}} D_{F_0}[\Phi],
\end{equation*}
where the second inequality holds by \eqref{J_bound}. Hence we obtain
\begin{equation*}
A_{F_0}[\Phi] \le \left(1+\frac{C_VC_{\kappa,L}^{P}}{c_{\kappa,L}} \right) D_{F_0}[\Phi].
\end{equation*}
Combining this with \eqref{J_bound}, we conclude that
\begin{equation*}
J_{F_0}[\Phi] + A_{F_0}[\Phi] \le \left( \frac{1}{c_{\kappa,L}} +1+ \frac{C_VC_{\kappa,L}^{P}}{c_{\kappa,L}} \right) D_{F_0}[\Phi] \leq \widetilde{C}_{\kappa,L} D_{F_0}[\Phi].
\end{equation*}
This proves \eqref{desired}, and hence \eqref{fourth3}.
\end{proof}

\subsection{Extension to perturbed potentials}\label{sec:potential-perturbation}

We now extend the inequalities proved for $F_0$ to those for perturbed potentials $F$ by comparing $Q_F$ with $Q_{F_0}$. The coefficients corresponding to $F_0$ are
\begin{equation}\label{eq:reference-coefficients}
q_0(y):=\kappa(y-u_-)(u_+-y),\quad V_0(y):=-\frac12\left(B_0''(y)(q_0)_y(y)+B_0'''(y)q_0(y)\right).
\end{equation}
By \eqref{eq:profile-coefficient-closeness}, the weighted spaces defined using $q$ and $q_0$ coincide, with equivalent norms. Thus, for any $\Phi$ satisfying the weighted integrability conditions in Proposition~\ref{fourth}, both $Q_F[\Phi]$ and $Q_{F_0}[\Phi]$ are well defined. Set
\begin{equation*}
D_{F_0}[\Phi]:=Q_{F_0}[\Phi]+\eta M[\Phi]^2.
\end{equation*}
The estimates in Proposition~\ref{fourth} and Corollary~\ref{fourth-cor}, proved for $F=F_0$ in Section~\ref{sec:quartic-completion}, imply
\begin{equation}\label{eq:reference-strong-coercivity}
J_{F_0}[\Phi]+\int_{u_-}^{u_+}q_0(y)^3\Phi_{yy}(y)^2\,dy+N[\Phi]+M[\Phi]^2\le CD_{F_0}[\Phi],
\end{equation}
where $C>0$ depends only on $\kappa$ and $L$. The following lemma estimates the difference between $Q_F[\Phi]$ and $Q_{F_0}[\Phi]$ in terms of $D_{F_0}[\Phi]$.

\begin{lemma}\label{lem:potential-functional-perturbation}
Under the assumptions in Lemma~\ref{lem:front-properties}, there exists a constant $C>0$, independent of $F$, such that
\begin{equation}\label{eq:potential-functional-perturbation} 
\left| Q_F[\Phi] - Q_{F_0}[\Phi] \right|\le C \rho_F D_{F_0}[\Phi] 
\end{equation}
for any $\Phi$ satisfying the weighted integrability conditions in Proposition~\ref{fourth}.
\end{lemma}

\begin{proof}
We first estimate the differences in the coefficients. By \eqref{eq:profile-coefficient-closeness}, we have
\begin{equation} \label{q3_diff}
|q(y)^3-q_0(y)^3|\le C\rho_Fq_0(y)^3,\quad y\in(u_-,u_+). 
\end{equation}
Rewriting $V$ as $V=-\frac12(B''q_y+B'''q)$ and recalling the definition of $V_0$ in \eqref{eq:reference-coefficients}, we have
\begin{equation*}
\begin{split} 
V-V_0=-\frac12\big( (B''-B_0'')q_y+B_0''(q_y-(q_0)_y)+(B'''-B_0''')q+B_0'''(q-q_0)\big). 
\end{split} 
\end{equation*}
Using \eqref{ppsize}, \eqref{eq:profile-coefficient-closeness}, and the boundedness of $B_0''$, $q_0$, and $(q_0)_y$ on $[u_-,u_+]$, we obtain
\begin{equation} \label{V_diff}
\|V-V_0\|_{L^\infty(u_-,u_+)}\le C\rho_F. 
\end{equation}
Therefore, by \eqref{q3_diff} and \eqref{V_diff},
\begin{equation*} 
\begin{split} 
|Q_F[\Phi]-Q_{F_0}[\Phi]| &= \left| \int_{u_-}^{u_+}(q^3-q_0^3)\Phi_{yy}^2\,dy+\int_{u_-}^{u_+}(V-V_0)\Phi^2\,dy \right| \\
& \leq C\rho_F\left(\int_{u_-}^{u_+}q_0^3\Phi_{yy}^2\,dy+N[\Phi]\right).
\end{split} 
\end{equation*}
Combining this estimate with \eqref{eq:reference-strong-coercivity} proves \eqref{eq:potential-functional-perturbation}.
\end{proof}

\begin{proof}[Completion of the proofs of Proposition~\ref{fourth} and Corollary~\ref{fourth-cor}]
Let $\Phi$ be as in Proposition~\ref{fourth}. By Lemma~\ref{lem:potential-functional-perturbation}, we obtain
\begin{equation*} 
\begin{split} 
Q_F[\Phi] + \eta M[\Phi]^2 &= D_{F_0}[\Phi] + Q_F[\Phi] - Q_{F_0}[\Phi] \\ 
&\ge (1-C\rho_F) D_{F_0}[\Phi] \ge \frac12 D_{F_0}[\Phi] 
\end{split} 
\end{equation*}
for sufficiently small $\rho_0>0$. Moreover, \eqref{eq:profile-coefficient-closeness} gives
\begin{equation*} 
cJ_{F_0}[\Phi] \le J_F[\Phi] \le CJ_{F_0}[\Phi],\quad \int_{u_-}^{u_+}q^3\Phi_{yy}^2 \, dy \le C\int_{u_-}^{u_+}q_0^3\Phi_{yy}^2 \, dy
\end{equation*}
for some constants $c>0$ and $C>0$. Using \eqref{eq:reference-strong-coercivity}, we obtain
\begin{equation*} 
Q_F[\Phi] + \eta M[\Phi]^2 \ge c_{\kappa,L} \left(J_F[\Phi]+M[\Phi]^2\right) 
\end{equation*}
and
\begin{equation*} 
J_F[\Phi] + \int_{u_-}^{u_+}q^3\Phi_{yy}^2 \, dy \le \widetilde C_{\kappa,L}\left(Q_F[\Phi]+\eta M[\Phi]^2\right). 
\end{equation*}
These are \eqref{fourth1} and \eqref{fourth3}, respectively. Finally, applying \eqref{fourth2} for $F=F_0$ and using the comparison of $J_F$ and $J_{F_0}$, we have
\begin{equation*} 
N[\Phi] \le C\left(J_{F_0}[\Phi]+M[\Phi]^2\right)\le C_{\kappa,L}\left(J_F[\Phi]+M[\Phi]^2\right). 
\end{equation*}
This proves \eqref{fourth2} and completes the proofs.
\end{proof}

\appendix

\section{\texorpdfstring{Technical justification of the \(L^2\)-framework}{Technical justification of the L2-framework}} \label{app:L2-framework}

In this appendix, we provide the technical justifications needed for the \(L^2\)-based argument. We treat the short-time \(C^\alpha\cap L^2\) perturbation class, the density argument for the \(L^2\)-a priori estimate, the front-coordinate transformation, and the Duhamel representation with a divergence-form source.

\subsection{\texorpdfstring{Short-time $C^\alpha\cap L^2$ perturbation class}{Short-time Calpha cap L2 perturbation class}} \label{app:L2-persistence}

We recall the short-time existence theory of \cite{Ho4} and refine the resulting $C^\alpha$-class by imposing the additional \(L^2\)-perturbation assumption \(u_0-\bar u\in L^2(\mathbb R)\), where $u_0:=u(\cdot,0)$. This yields a short-time \(C^\alpha\cap L^2\) perturbation class with persistence and continuous dependence in the \(L^2\) norm.

\begin{lemma}[Short-time \(C^\alpha\)-class] \label{lem:Calpha-local-class}
Let \(R_0>0\). There exist constants \(T>0\) and \(R>0\), depending only on \(R_0\), such that for initial data \(u_0\) satisfying
\begin{equation*}
\|u_0\|_{C^\alpha(\mathbb R)}\le R_0,
\end{equation*}
the Cauchy problem for \eqref{eq:gCH} has a unique solution \(u\in\mathcal C_{T,R}^{\alpha}\), where \(\mathcal C_{T,R}^{\alpha}\) denotes the class of solutions satisfying
\begin{equation*}
u\in C^{\alpha,\alpha/4}(\mathbb R\times[0,T]), \quad \|u\|_{C^{\alpha,\alpha/4}(\mathbb R\times[0,T])}\le R,
\end{equation*}
and, for each \(\sigma\in(0,T)\),
\begin{equation*}
u\in C^{4+\alpha,1+\alpha/4}(\mathbb R\times[\sigma,T]).
\end{equation*}
\end{lemma}

\begin{proof}
Since $B=F'\in C^3(\mathbb R)$, this is obtained by a direct application of the short-time existence theorem \cite[Theorem~1.1]{Ho4} with \(2p=4\). The uniform choice of \(T\) and \(R\) follows from the contraction argument in the short-time construction.
\end{proof}

We now consider perturbations of the front \(\bar u\) in the \(C^\alpha\)-class, under the additional assumption
\begin{equation*}
\phi_0:=u_0-\bar u\in L^2(\mathbb R).
\end{equation*}
As in Section~\ref{sec2}, we define the shift function \(\delta\) and the perturbation \(\phi\) associated with \(u\in\mathcal C_{T,R}^{\alpha}\) by
\begin{equation} \label{delta_u}
\dot\delta(t)=-\eta\int_{\mathbb R}\bar u'(x)\big(u(x+\delta(t),t)-\bar u(x)\big)\,dx,\quad \delta(0)=0,
\end{equation}
and
\begin{equation} \label{phi_u}
\phi(x,t):=u(x+\delta(t),t)-\bar u(x).
\end{equation}
For \(t>0\), the perturbation $\phi$ satisfies
\begin{equation*}
\phi_t=-\phi_{xxxx}+\partial_x^2\big(B(\bar u+\phi)-B(\bar u)\big)+\dot\delta(\bar u'+\phi_x).
\end{equation*}
Let \(S(t):=e^{-t\partial_x^4}\) be the \(C_0\)-semigroup on \(L^2(\mathbb R)\) generated by \(-\partial_x^4\). Then, by Duhamel's formula,
\begin{equation} \label{Duh}
\begin{split}
\phi(t) &= S(t)\phi_0+\int_0^t\partial_x^2S(t-s)\big(B(\bar u+\phi(s))-B(\bar u)\big)\,ds \\
& \quad +\int_0^t\dot\delta(s)S(t-s)\bar u'\,ds+\int_0^t\dot\delta(s)\partial_xS(t-s)\phi(s)\,ds.
\end{split}
\end{equation}

\begin{proposition}[Short-time \(C^\alpha\cap L^2\)-class of perturbations] \label{prop:Calpha-L2-class}
Let \(R_0>0\) and \(M_0>0\). There exist constants \(T>0\), \(R>0\), and $C>0$, depending only on \(R_0\) and \(M_0\), such that the following statement holds.

Let $u_0$ be initial data satisfying
\begin{equation*}
\lVert u_0 \rVert_{C^\alpha(\mathbb{R})} \leq R_0, \quad \lVert \phi_0 \rVert_{L^2(\mathbb{R})} \leq M_0,
\end{equation*}
where $\phi_0 := u_0 - \bar{u}$. Let \(u \in\mathcal C_{T,R}^{\alpha}\) be the corresponding local solution given by Lemma~\ref{lem:Calpha-local-class}. Then the perturbation $\phi$ and the shift $\delta$, defined in \eqref{phi_u} and \eqref{delta_u}, respectively, satisfy
\begin{equation*}
\phi \in C ([0,T];L^2(\mathbb{R})), \quad \delta \in C^1([0,T]).
\end{equation*}
Moreover, if \(u^j\in\mathcal C^\alpha_{T,R}\), \(j=1,2\), are the local solutions with initial data \(u_0^j\) satisfying
\begin{equation*}
\|u_0^j\|_{C^\alpha(\mathbb R)}\le R_0,\quad \|u_0^j-\bar u\|_{L^2(\mathbb R)}\le M_0,\quad j=1,2,
\end{equation*}
then the perturbations \(\phi^j\) and shifts \(\delta^j\), defined by \eqref{phi_u} and \eqref{delta_u}, satisfy
\begin{equation*}
\|\phi^1-\phi^2\|_{C([0,T];L^2(\mathbb R))}+\|\delta^1-\delta^2\|_{C^1([0,T])}\le C\|u_0^1-u_0^2\|_{L^2(\mathbb R)}.
\end{equation*}
\end{proposition}

\begin{proof}
For \(v\in C([0,T];L^2(\mathbb R))\) satisfying the prescribed H\"older bound \(\|\bar u+v\|_{C^{\alpha,\alpha/4}(\mathbb R\times[0,T])}\le R\), set
\begin{equation*}
\dot\delta_v(t):=-\eta\int_{\mathbb R}\bar u'(x)v(x,t)\,dx
\end{equation*}
and denote by \(\mathcal T v\) the right-hand side of \eqref{Duh} with \(\phi\) and \(\dot\delta\) replaced by \(v\) and \(\dot\delta_v\), respectively. After enlarging \(R\) and reducing \(T\) if necessary, the standard H\"older estimates for \eqref{Duh} show that \(\mathcal T\) preserves the prescribed \(C^{\alpha,\alpha/4}\)-bound on the trial functions \(\bar u+v\) satisfying \(\|v\|_{C([0,T];L^2)}\le 2M_0\). For such \(v\), the \(C^\alpha\)-bound and the local Lipschitz continuity of \(B\) give
\begin{equation*}
\|B(\bar u+v)-B(\bar u)\|_{L^2}\le C_R\|v\|_{L^2}.
\end{equation*}
The shift equation \eqref{delta_u} gives
\begin{equation*}
|\dot\delta_v(t)|\le \eta\|\bar u'\|_{L^2}\|v(\cdot,t)\|_{L^2}.
\end{equation*}
Using the Duhamel formula \eqref{Duh} and the semigroup estimates
\begin{equation*}
\|\partial_x^2S(t)f\|_{L^2}\le Ct^{-1/2}\|f\|_{L^2},\quad \|\partial_xS(t)f\|_{L^2}\le Ct^{-1/4}\|f\|_{L^2},
\end{equation*}
we obtain
\begin{equation*}
\begin{split}
\|\mathcal T v(t)\|_{L^2} &\le \|\phi_0\|_{L^2}+C_R\int_0^t(t-s)^{-1/2}\|v(s)\|_{L^2}\,ds \\
&\quad + C\int_0^t\|v(s)\|_{L^2}\,ds+C\int_0^t(t-s)^{-1/4}\|v(s)\|_{L^2}^2\,ds.
\end{split}
\end{equation*}
Thus the map $\mathcal{T}$ preserves the bound \(\|v\|_{C([0,T];L^2)}\le 2M_0\). The same estimates applied to the difference of two Duhamel formulas give a contraction in the \(L^2\)-component for sufficiently small \(T>0\). Hence \(\mathcal T\) has a unique fixed point in the \(L^2\)-component under the prescribed \(L^2\)-bound. By uniqueness in the \(C^\alpha\)-class, this fixed point agrees with the perturbation associated with the solution in the \(C^\alpha\)-class. Therefore $\phi\in C([0,T];L^2(\mathbb R))$. Moreover, since \(\bar u'\in L^2(\mathbb R)\), the shift equation \eqref{delta_u} gives \(\dot\delta\in C([0,T])\), and hence \(\delta\in C^1([0,T])\).

Finally, applying the same contraction estimate to the difference of two Duhamel formulas, retaining the initial difference, yields the stated continuous dependence estimate for \(\phi^1-\phi^2\), after reducing \(T>0\) if necessary. The estimate for \(\delta^1-\delta^2\) in \(C^1([0,T])\) follows directly from \eqref{delta_u}, since
\begin{equation*}
|\dot\delta^1(t)-\dot\delta^2(t)|\le \eta\|\bar u'\|_{L^2}\|\phi^1(\cdot,t)-\phi^2(\cdot,t)\|_{L^2}
\end{equation*}
and \(\delta^1(0)=\delta^2(0)=0\).
\end{proof}

\subsection{\texorpdfstring{Density argument in the a priori estimate}{Density argument in the a priori estimate}} \label{app:smooth-approx}

We provide the density argument invoked in the proof of Proposition~\ref{apriori}. Let \(\varepsilon_1>0\) be the constant obtained in Section~\ref{sec2.3} for smooth compactly supported perturbations. We first prove the estimate for admissible perturbations satisfying
\begin{equation*}
\sup_{0\le t\le T_0}\|\phi(\cdot,t)\|_{L^2(\mathbb R)}\le \varepsilon_2,
\end{equation*}
where \(\varepsilon_2>0\) is chosen so that
\begin{equation*}
2\varepsilon_2\le \varepsilon_1.
\end{equation*}
At the end of the argument, we rename \(\varepsilon_2\) as \(\varepsilon_1\) in Proposition~\ref{apriori}. We organize the argument into four steps.

\emph{Step 1: Approximation of the initial data.}
Let \(\phi\) be an admissible perturbation on \([0,T_0]\), and assume that
\begin{equation*}
\sup_{0\le t\le T_0}\|\phi(\cdot,t)\|_{L^2(\mathbb R)}\le\varepsilon_2.
\end{equation*}
Setting \(\phi_0:=\phi(\cdot,0)\), by Definition~\ref{def:admissible}, \(\phi_0\in C^\alpha(\mathbb R)\cap L^2(\mathbb R)\). Fix \(\beta\in(0,\alpha)\). Choose \(\phi_0^n\in C_c^\infty(\mathbb R)\) such that
\begin{equation*}
\phi_0^n\to\phi_0\quad\text{in }L^2(\mathbb R)\cap C^\beta(\mathbb R),
\end{equation*}
and such that the corresponding data \(u_0^n:=\bar u+\phi_0^n\) remain in a bounded subset of \(C^\alpha(\mathbb R)\). This can be achieved by a standard cutoff and mollification argument. Hence there exist \(R_0>0\) and \(M_0>0\), independent of \(n\), such that
\begin{equation*}
\|u_0^n\|_{C^\alpha(\mathbb R)}+\|u_0\|_{C^\alpha(\mathbb R)}\le R_0,\quad
\|\phi_0^n\|_{L^2(\mathbb R)}+\|\phi_0\|_{L^2(\mathbb R)}\le M_0.
\end{equation*}

\emph{Step 2: Construction and convergence of the approximate solutions.}
For each \(n\), let \(u^n\) denote the solution with initial data \(u_0^n\), and let \(\delta^n\) and \(\phi^n\) be the associated shift and perturbation defined by \eqref{delta_u} and \eqref{phi_u}. Since \(u_0^n\to u_0\) in \(C^\beta(\mathbb R)\), local \(C^\beta\)-well-posedness and continuous dependence first give convergence on a short time interval. Repeating this argument along a finite partition of \([0,T_0]\), and using that the admissible solution \(u\) is defined on the whole interval, we obtain that, for all sufficiently large \(n\), the solutions \(u^n\) exist on \([0,T_0]\) and remain uniformly bounded in \(C^{\beta,\beta/4}(\mathbb R\times[0,T_0])\). Applying Proposition~\ref{prop:Calpha-L2-class}, with \(\beta\) in place of \(\alpha\), on the same partition gives
\begin{equation*}
\|\phi^n-\phi\|_{C([0,T_0];L^2(\mathbb R))}\le C(T_0)\|\phi_0^n-\phi_0\|_{L^2(\mathbb R)}\to0.
\end{equation*}
Moreover, by the shift equation \eqref{delta_u},
\begin{equation*}
\|\dot\delta^n-\dot\delta\|_{C([0,T_0])}\le \eta\|\bar u'\|_{L^2(\mathbb R)}\|\phi^n-\phi\|_{C([0,T_0];L^2(\mathbb R))}\to0.
\end{equation*}
Since \(\delta^n(0)=\delta(0)=0\), this also gives \(\delta^n\to\delta\) in \(C^1([0,T_0])\). Therefore, for all sufficiently large \(n\),
\begin{equation*}
\sup_{0\le t\le T_0}\|\phi^n(\cdot,t)\|_{L^2(\mathbb R)} \le \sup_{0\leq t \leq T_0} \lVert \phi(\cdot,t) \rVert_{L^2} + \lVert \phi^n-\phi \rVert_{C([0,T_0];L^2)} \le2\varepsilon_2\le\varepsilon_1.
\end{equation*}

\emph{Step 3: Application of the smooth a priori estimate.}
For such \(n\), the perturbation \(\phi^n\) is generated by a smooth compactly supported initial perturbation. Hence the integrations by parts in Section~\ref{sec2.3}, including the front-coordinate calculation, are justified for \(\phi^n\). The a priori estimate for smooth perturbations gives
\begin{equation*}
\|\phi^n(\cdot,t)\|_{L^2}^2+c_0\int_0^t \left( \|\phi_x^n(\cdot,s)\|_{H^1}^2+\|\sqrt{\bar u'}\,\phi^n(\cdot,s)\|_{L^2}^2+|\dot\delta^n(s)|^2 \right) \,ds\le \|\phi_0^n\|_{L^2}^2
\end{equation*}
for all \(t\in[0,T_0]\).

\emph{Step 4: Passage to the limit.}
Fix \(t\in[0,T_0]\). The endpoint terms converge by the strong convergence in \(C([0,T_0];L^2)\). The uniform bound on \(\int_0^{T_0}\mathcal{D}_n(s)\,ds\) implies, after extracting a subsequence if necessary, that \(\phi_x^n\) converges weakly in \(L^2(0,T_0;H^1(\mathbb R))\). Since \(\phi^n\to\phi\) strongly in \(L^2(0,T_0;L^2(\mathbb R))\), this weak limit is \(\phi_x\) in the sense of distributions. Hence \(\phi_x\in L^2(0,T_0;H^1(\mathbb R))\), and lower semicontinuity implies
\begin{equation*}
\int_0^t\|\phi_x(\cdot,s)\|_{H^1}^2\,ds\le\liminf_{n\to\infty}\int_0^t\|\phi_x^n(\cdot,s)\|_{H^1}^2\,ds.
\end{equation*}
Since \(\sqrt{\bar u'}\in L^\infty(\mathbb R)\), we have
\begin{equation*}
\sqrt{\bar u'}\,\phi^n\to\sqrt{\bar u'}\,\phi\quad\text{in }L^2([0,T_0];L^2(\mathbb R)).
\end{equation*}
Together with \(\dot\delta^n\to\dot\delta\) in \(C([0,T_0])\), this yields
\begin{equation*}
\int_0^t \mathcal{D} (s)\,ds\le\liminf_{n\to\infty}\int_0^t \mathcal{D}_n(s)\,ds.
\end{equation*}
Letting \(n\to\infty\), we obtain
\begin{equation*}
\|\phi(\cdot,t)\|_{L^2}^2+c_0\int_0^t \mathcal{D}(s)\,ds\le \|\phi(\cdot,0)\|_{L^2}^2.
\end{equation*}
Since \(t\in[0,T_0]\) was arbitrary, the estimate holds on \([0,T_0]\). In particular,
\begin{equation*}
\phi_x\in L^2([0,T_0];H^1(\mathbb R)),\quad \sqrt{\bar u'}\,\phi\in L^2([0,T_0];L^2(\mathbb R)),\quad \dot\delta\in L^2([0,T_0]).
\end{equation*}
Thus the quantities appearing in \(\mathcal{D}(t)\) are well-defined for a.e. \(t\in[0,T_0]\).

\subsection{The front-coordinate transformation} \label{app:front-coordinate}

We justify the use of the front coordinate in the energy estimates. Suppose that \(\phi\in H^2(\mathbb R)\), and set
\begin{equation*}
y=\bar u(x),\quad\Phi(y)=\phi(x),\quad q(y)=\bar u'(x)=\sqrt{2F(y)}.
\end{equation*}
Since $dy=q(y)\,dx$ and $\bar u'$ is bounded by Lemma~\ref{lem:front-properties}, we have
\begin{equation*}
\int_{u_-}^{u_+}\Phi^2\,dy=\int_{\mathbb R}\bar u'\phi^2\,dx\le C\|\phi\|_{L^2}^2, \quad \int_{u_-}^{u_+}q(y)\Phi_y(y)^2\,dy=\int_{\mathbb R}\phi_x(x)^2\,dx.
\end{equation*}
Moreover,
\begin{equation*}
\phi_{xx}=B(y)\Phi_y+q(y)^2\Phi_{yy}.
\end{equation*}
Using \(B(y)=q(y)q_y(y)\) and the boundedness of \(q_y\), we obtain
\begin{equation*}
\int_{u_-}^{u_+}q(y)^3\Phi_{yy}(y)^2\,dy\le C\|\phi_{xx}\|_{L^2(\mathbb R)}^2+C\|\phi_x\|_{L^2(\mathbb R)}^2.
\end{equation*}
Therefore, by \eqref{eq:profile-coefficient-closeness}, the normalized function \(\Psi(z)=\Phi(y_c+\frac{L}{2}z)\), where \(y_c=\frac{u_-+u_+}{2}\), belongs to the weighted space \(X\) defined in Appendix~\ref{density}. More precisely,
\begin{equation*}
\|\Psi\|_X\le C\|\phi\|_{H^2(\mathbb R)}.
\end{equation*}

It remains to justify that the endpoint contribution arising from the integration by parts in the front coordinate vanishes. Since \(B=q q_y\), we have
\begin{equation*}
B(y)q(y)\Phi_y(y)^2=q_y(y)\big(q(y)\Phi_y(y)\big)^2=q_y(\bar u(x))\phi_x(x)^2.
\end{equation*}
Here \(y\to u_\pm\) corresponds to \(x\to\pm\infty\). Moreover, \(\phi\in H^2(\mathbb R)\) implies \(\phi_x\in H^1(\mathbb R)\), so the continuous representative of \(\phi_x\) vanishes at spatial infinity. Hence
\begin{equation*}
\lim_{y\to u_\pm}B(y)q(y)\Phi_y(y)^2=0.
\end{equation*}
Thus \([Bq\Phi_y^2]_{u_-}^{u_+}=0\) whenever \(\phi(\cdot,t)\in H^2(\mathbb R)\). This endpoint cancellation is first justified for the smooth approximants. The resulting energy estimate, obtained using the front-coordinate integration by parts, is then passed to the limiting admissible solution by the density argument above. Since the dissipation estimate gives \(\phi_x\in L^2(0,T_0;H^1(\mathbb R))\), we have \(\lim_{y\to u_\pm}B(y)q(y)\Phi_y(y,t)^2=0\) for a.e. \(t\in(0,T_0)\). For classical solutions, parabolic regularization gives this for all \(t>0\).

\subsection{Justification of the Duhamel representation} \label{app:duhamel-justification}

We justify the Duhamel representation used in Section~\ref{sec3} for the source \(\mathcal F(\phi)\). By Lemma~\ref{lem:source-tail-asymp}, for any compact interval \(I=[T,T_1]\subset(0,\infty)\), it holds that
\begin{equation*}
\mathcal F(\phi)\in L^2(I;L^2(\mathbb R)).
\end{equation*}
Choose \(F^m\in C_c^\infty(I\times\mathbb R)\) such that \(F^m\to \mathcal F(\phi)\) in \(L^2(I;L^2)\), and set
\begin{equation*}
w^m(t):=\int_T^t e^{(t-s)\mathcal A_\infty}\partial_xF^m(s)\,ds.
\end{equation*}
Then \(w^m\) is the classical solution of
\begin{equation*}
w_t^m=\mathcal A_\infty w^m+\partial_xF^m,\quad w^m(T)=0.
\end{equation*}
Applying the energy estimate from Section~\ref{sec3.1.2} to \(w^m-w^\ell\) gives
\begin{equation*}
\sup_{T\le t\le T_1}\|w^m(t)-w^\ell(t)\|_{L^2}\le \frac{1}{\kappa L}\|F^m-F^\ell\|_{L^2(I;L^2)}.
\end{equation*}
Hence \(w^m\) converges in \(C(I;L^2(\mathbb R))\), independently of the approximating sequence. We define
\begin{equation*}
\int_T^t e^{(t-s)\mathcal A_\infty}\partial_x\mathcal F(\phi(s))\,ds :=\lim_{m\to\infty}\int_T^t e^{(t-s)\mathcal A_\infty}\partial_xF^m(s)\,ds,
\end{equation*}
where the limit is taken in \(C(I;L^2(\mathbb R))\). The same estimate yields
\begin{equation*}
\sup_{T\le t\le T_1}\left\|\int_T^t e^{(t-s)\mathcal A_\infty}\partial_x\mathcal F(\phi(s))\,ds\right\|_{L^2}\le \frac{1}{\kappa L}\|\mathcal F(\phi)\|_{L^2(I;L^2)}.
\end{equation*}
Let \(\tilde\phi\) denote the right-hand side of the Duhamel formula, namely
\begin{equation*}
\tilde\phi(t):=e^{(t-T)\mathcal A_\infty}\phi(\cdot,T)+\int_T^t e^{(t-s)\mathcal A_\infty}\partial_x\mathcal F(\phi(\cdot,s))\,ds.
\end{equation*}
By the construction above, \(\tilde\phi\) is the weak solution of the inhomogeneous problem with source \(\partial_x\mathcal F(\phi)\) and initial data \(\phi(\cdot,T)\). Since \(\phi\) satisfies the same problem in the distributional sense, the energy estimate applied to \(\phi-\tilde\phi\) gives uniqueness, and hence
\begin{equation*}
\phi(\cdot,t)=e^{(t-T)\mathcal A_\infty}\phi(\cdot,T)+\int_T^t e^{(t-s)\mathcal A_\infty}\partial_x\mathcal F(\phi(\cdot,s))\,ds
\end{equation*}
in \(L^2(\mathbb R)\). The identity \(\langle w,\partial_x\mathcal F(\phi)\rangle=-\langle w_x,\mathcal F(\phi)\rangle\), used in the energy estimate, is first justified for the smooth sources \(F^m\) and then passed to the limit by the same approximation argument.

\section{\texorpdfstring{Density of polynomials in the weighted space $X$}{Density of polynomials in the weighted space X}} \label{density}
Recall the weighted space $X$ appearing in the proof of Lemma~\ref{CH-Poincare}. More explicitly, we define
\begin{equation*}
X := \left\{\Psi\in L^2(-1,1): (1-z^2)^{1/2}\Psi_z\in L^2(-1,1),\  (1-z^2)^{3/2}\Psi_{zz}\in L^2(-1,1) \right\},
\end{equation*}
where the derivatives are understood in the sense of distributions, and we set
\begin{equation*}
\|\Psi\|_X^2 := \int_{-1}^1\Psi^2\,dz + \int_{-1}^1(1-z^2)\Psi_z^2\,dz + \int_{-1}^1(1-z^2)^3\Psi_{zz}^2\,dz.
\end{equation*}
In this appendix, we prove that polynomials are dense in $X$ with respect to $\|\cdot\|_X$.

Let $\Psi\in X$. We make the change of variables
\begin{equation*}
z=\tanh s, \quad f(s)=\Psi(\tanh s), \quad dz=\operatorname{sech}^2s\,ds, \quad 1-z^2=\operatorname{sech}^2s.
\end{equation*}
Using
\begin{equation*}
f_s=(1-z^2)\Psi_z, \quad f_{ss}=(1-z^2)^2\Psi_{zz}-2z(1-z^2)\Psi_z,
\end{equation*}
where $z=\tanh s$, we obtain
\begin{equation} \label{norm_id}
\|\Psi\|_X^2=\int_{\mathbb R} f^2\operatorname{sech}^2s\,ds+\int_{\mathbb R}f_s^2\,ds+\int_{\mathbb R}(f_{ss}+2(\tanh s)f_s)^2\,ds.
\end{equation}
This identity gives $f \in L_{\mathrm{loc}}^2(\mathbb{R})$, $f_s\in L^2(\mathbb R)$, and $f_{ss}+2(\tanh s)f_s\in L^2(\mathbb R)$. Since $\tanh s$ is bounded, it follows that $f_{ss}\in L^2(\mathbb R)$, and hence $f_s\in H^1(\mathbb R)$.

By the density of $C_c^\infty(\mathbb R)$ in $H^1(\mathbb R)$, we may choose a sequence $g_k\in C_c^\infty(\mathbb R)$ such that
\begin{equation*}
g_k\to f_s \quad \text{in } H^1(\mathbb R).
\end{equation*}
Define
\begin{equation*}
f_k(s):=f(0)+\int_0^s g_k(r)\,dr.
\end{equation*}
Then
\begin{equation*}
(f_k)_s\to f_s \quad \text{in } L^2(\mathbb R), \quad (f_k)_{ss}+2(\tanh s)(f_k)_s\to f_{ss}+2(\tanh s)f_s \quad \text{in } L^2(\mathbb R).
\end{equation*}
Moreover, since $f_k(0)=f(0)$, the Cauchy--Schwarz inequality gives
\begin{equation*}
|f_k(s)-f(s)|^2 = \left| \int_0^s(g_k(r)-f_s(r))\,dr \right|^2 \le |s|\|g_k-f_s\|_{L^2(\mathbb R)}^2.
\end{equation*}
Multiplying by $\operatorname{sech}^2s$ and integrating, we obtain
\begin{equation*}
\int_{\mathbb R}|f_k-f|^2\operatorname{sech}^2s\,ds\le \left(\int_{\mathbb R}|s|\operatorname{sech}^2s\,ds\right)\|g_k-f_s\|_{L^2(\mathbb R)}^2\to0.
\end{equation*}

Now we set
\begin{equation*}
\Psi_k(z):=f_k(\operatorname{arctanh}z).
\end{equation*}
Since $g_k$ is compactly supported, each $f_k$ is constant for sufficiently large positive and negative $s$. Hence $\Psi_k$ extends to a function in $C^\infty([-1,1])$. By the norm identity \eqref{norm_id} and the convergence of $f_k$ in the transformed norm, we have $\Psi_k\to\Psi$ in $X$.

Fix $k$. Applying the Weierstrass approximation theorem to $\Psi_k''$ and integrating twice, with the integration constants chosen to match $\Psi_k(-1)$ and $\Psi_k'(-1)$, gives polynomials $p_{k,j}$ such that $p_{k,j}\to\Psi_k$ in $C^2([-1,1])$. Since the weights in the $X$-norm are bounded on $[-1,1]$, this convergence implies
\begin{equation*}
\|p_{k,j}-\Psi_k\|_X \le C\|p_{k,j}-\Psi_k\|_{C^2([-1,1])}\to0.
\end{equation*}
Choosing $j=j(k)$ such that $\|p_{k,j(k)}-\Psi_k\|_X\le k^{-1}$ and setting $p_k:=p_{k,j(k)}$, we obtain
\begin{equation*}
\|p_k-\Psi\|_X\le k^{-1}+\|\Psi_k-\Psi\|_X\to0.
\end{equation*}
Therefore polynomials are dense in $X$.

\section{\texorpdfstring{Legendre coefficient formulas for \(Q_0\)}{Legendre coefficient formulas for Q0}} \label{app:Q0-coefficients}

In this appendix, we provide a detailed derivation of the Legendre coefficient formulas used in the proof of Lemma~\ref{CH-Poincare}. Let $P_n$ be the Legendre polynomial normalized by $P_n(1)=1$. Then it holds that
\begin{equation*}
-\big((1-z^2)P_n'\big)'=n(n+1)P_n
\end{equation*}
and
\begin{equation}\label{app:Legendre-orth}
\int_{-1}^1P_nP_m\,dz=h_n\delta_{nm},\quad h_n:=\frac{2}{2n+1}.
\end{equation}
For a polynomial $\Psi$, we write its Legendre expansion as
\begin{equation*}
\Psi(z)=\sum_{n\ge0}a_nP_n(z).
\end{equation*}
We express the quadratic form
\begin{equation} \label{Q00}
Q_0[\Psi]=\frac18\int_{-1}^{1}(1-z^2)^3\Psi_{zz}^2\,dz + \frac34\int_{-1}^{1}(3z^2-1)\Psi^2\,dz
\end{equation}
in terms of the Legendre coefficients $a_n$. More precisely, we prove that
\begin{equation} \label{Q0form}
Q_0[\Psi]=\sum_{n\ge0}d_na_n^2+2\sum_{n\ge0}e_na_na_{n+2},
\end{equation}
where $d_n$ and $e_n$ are given by \eqref{dnen} in the proof of Lemma~\ref{CH-Poincare}.

We first consider the fourth-order contribution, i.e., the first integral on the right-hand side of \eqref{Q00}:
\begin{equation*}
\frac18\int_{-1}^1(1-z^2)^3\Psi_{zz}^2\,dz=\frac18\sum_{n,m\ge0}a_na_m\int_{-1}^1(1-z^2)^3P_n''P_m''\,dz.
\end{equation*}
We define
\begin{equation*}
\mathcal L P_n:=\big((1-z^2)^3P_n''\big)''.
\end{equation*}
The operator $\mathcal L$ exhibits several structural properties that are well aligned with the Legendre basis. First, for each $n\ge0$, $\mathcal L P_n$ is a polynomial of degree at most $n+2$. Moreover, $\mathcal L$ is symmetric on Legendre polynomials. Indeed, since $(1-z^2)^3$ and its first derivative vanish at $z=\pm1$, integration by parts twice gives
\begin{equation}\label{app:L-self-adjoint}
\int_{-1}^1(\mathcal L P_m)P_n\,dz=\int_{-1}^1(1-z^2)^3P_m''P_n''\,dz=\int_{-1}^1(\mathcal L P_n)P_m\,dz.
\end{equation}
Finally, $\mathcal L$ preserves parity, since $(1-z^2)^3$ is even and differentiation twice preserves evenness and oddness.

These properties impose a restriction on the Legendre expansion of $\mathcal L P_n$: only the modes $P_{n-2}$, $P_n$, and $P_{n+2}$ may appear. To see this, fix $n$. Since $\deg(\mathcal L P_n)\le n+2$, the expansion of $\mathcal L P_n$ can involve only $P_0,\dots,P_{n+2}$. We now show that the coefficients of $P_m$ vanish for all $m\le n-4$. If $m\le n-4$, then $\deg(\mathcal L P_m)\le m+2\le n-2$. Hence $\mathcal L P_m$ is a polynomial of degree at most $n-2$, and therefore belongs to $\operatorname{span}\{P_0,\dots,P_{n-2}\}$. By the orthogonality \eqref{app:Legendre-orth} and the symmetry \eqref{app:L-self-adjoint}, we have
\begin{equation*}
\int_{-1}^1(\mathcal L P_n)P_m\,dz=\int_{-1}^1P_n(\mathcal L P_m)\,dz=0 \quad \text{for } m \leq n-4.
\end{equation*}
Thus the expansion of $\mathcal{L} P_n$ can contain only the modes $P_{n+2},P_{n+1},P_n,P_{n-1},P_{n-2}$. Since $P_n(-z)=(-1)^nP_n(z)$, the Legendre polynomial $P_n$ is even when $n$ is even and odd when $n$ is odd. Recalling that $\mathcal L$ preserves parity, the modes $P_{n-1}$ and $P_{n+1}$ have the opposite parity to $\mathcal L P_n$, so their coefficients vanish. We may write $\mathcal{L}P_n$ in the form
\begin{equation}\label{app:L-three-term}
\mathcal L P_n=\alpha_nP_{n+2}+\beta_nP_n+\gamma_nP_{n-2},
\end{equation}
with the convention that $\gamma_nP_{n-2}=0$ for $n=0,1$.

We now determine the coefficients $\alpha_n$, $\beta_n$, and $\gamma_n$ in \eqref{app:L-three-term}. Let $\ell_n$ denote the coefficient of $z^n$ in $P_n$. From the three-term recurrence relation
\begin{equation} \label{recur}
(n+1)P_{n+1}(z)=(2n+1)zP_n(z)-nP_{n-1}(z),
\end{equation}
we obtain, by comparing coefficients,
\begin{equation*}
(n+1)\ell_{n+1}=(2n+1)\ell_n.
\end{equation*}
Since $\ell_0=1$, it follows that
\begin{equation*}
\ell_n=\prod_{k=0}^{n-1}\frac{2k+1}{k+1}=\frac{(2n)!}{2^n(n!)^2}.
\end{equation*}
Consequently,
\begin{equation*}
\frac{\ell_n}{\ell_{n+2}}=\frac{2^{-n}(2n)!/(n!)^2}{2^{-(n+2)}(2n+4)!/((n+2)!)^2}=\frac{(n+1)(n+2)}{(2n+1)(2n+3)}.
\end{equation*}
The coefficient of $z^{n+2}$ in $\mathcal L P_n$ is $-n(n-1)(n+3)(n+4)\ell_n$. Therefore $\alpha_n$ is given by
\begin{equation*}
\alpha_n=-\frac{(n-1)n(n+1)(n+2)(n+3)(n+4)}{(2n+1)(2n+3)}.
\end{equation*}
By symmetry \eqref{app:L-self-adjoint}, applied to the pair $(P_n,P_{n-2})$, and using \eqref{app:Legendre-orth}, we have
\begin{equation*}
\gamma_nh_{n-2}=\alpha_{n-2}h_n, \quad n\ge2,
\end{equation*}
and hence
\begin{equation*}
\gamma_n=-\frac{(n-3)(n-2)(n-1)n(n+1)(n+2)}{(2n-1)(2n+1)}, \quad n\ge2.
\end{equation*}
Finally, since $\mathcal L P_n(1)=0$ and $P_k(1)=1$ for all $k\ge0$, we have $\alpha_n+\beta_n+\gamma_n=0$. Thus
\begin{equation*}
\beta_n=-\alpha_n-\gamma_n=\frac{2n(n-1)(n+1)(n+2)(n^2+n+3)}{(2n-1)(2n+3)}.
\end{equation*}
Using \eqref{app:L-self-adjoint}, \eqref{app:L-three-term}, and \eqref{app:Legendre-orth}, we obtain
\begin{equation*}
\begin{split}
\int_{-1}^1(1-z^2)^3P_n''P_m''\,dz &=\int_{-1}^1(\mathcal L P_n)P_m \, dz \\
&=\alpha_nh_{n+2}\delta_{m,(n+2)}+\beta_nh_n\delta_{mn}+\gamma_nh_{n-2}\delta_{m,(n-2)}.
\end{split}
\end{equation*}
In particular, taking $m=n$ yields
\begin{equation*}
\int_{-1}^1(1-z^2)^3(P_n'')^2\,dz=\beta_nh_n.
\end{equation*}
Therefore
\begin{equation}\label{app:fourth-diag}
\frac18\int_{-1}^1(1-z^2)^3(P_n'')^2\,dz=\frac18\beta_nh_n=\frac{n(n-1)(n+1)(n+2)(n^2+n+3)}{2(2n-1)(2n+1)(2n+3)}.
\end{equation}
Taking $m=n+2$ in the same identity gives
\begin{equation*}
\int_{-1}^1(1-z^2)^3P_n''P_{n+2}''\,dz=\alpha_nh_{n+2}.
\end{equation*}
Since $h_{n+2}=\frac{2}{2n+5}$, we obtain
\begin{equation}\label{app:fourth-cross}
\frac18\int_{-1}^1(1-z^2)^3P_n''P_{n+2}''\,dz = \frac18\alpha_nh_{n+2} =-\frac{(n-1)n(n+1)(n+2)(n+3)(n+4)}{4(2n+1)(2n+3)(2n+5)}.
\end{equation}

We next consider the zeroth-order contribution:
\begin{equation*}
\frac34\int_{-1}^1(3z^2-1)\Psi^2 \,dz=\frac34\sum_{n,m\ge0}a_na_m\int_{-1}^1(3z^2-1)P_nP_m \,dz.
\end{equation*}
Applying the recurrence relation \eqref{recur} twice, we have
\begin{equation*}
z^2 P_n = \frac{n+1}{2n+1}zP_{n+1} + \frac{n}{2n+1}zP_{n-1} = \rho_nP_{n+2}+\sigma_nP_n+\tau_nP_{n-2},
\end{equation*}
where
\begin{equation*}
\rho_n=\frac{(n+1)(n+2)}{(2n+1)(2n+3)},\quad \sigma_n=\frac{2n^2+2n-1}{(2n-1)(2n+3)},\quad \tau_n=\frac{n(n-1)}{(2n-1)(2n+1)}.
\end{equation*}
Using \eqref{app:Legendre-orth}, we obtain
\begin{equation*}
\int_{-1}^1(3z^2-1)P_nP_m\,dz = 3\rho_nh_{n+2}\delta_{m,(n+2)}+(3\sigma_n-1)h_n\delta_{mn}+3\tau_nh_{n-2}\delta_{m,(n-2)}.
\end{equation*}
In particular, taking $m=n$ gives
\begin{equation}\label{app:potential-diag}
\frac34\int_{-1}^1(3z^2-1)P_n^2\,dz = \frac34(3\sigma_n-1)h_n=\frac{3n(n+1)}{(2n-1)(2n+1)(2n+3)}.
\end{equation}
Taking $m=n+2$ gives
\begin{equation}\label{app:potential-cross}
\frac34\int_{-1}^1(3z^2-1)P_nP_{n+2}\,dz = \frac94\rho_nh_{n+2}=\frac{9(n+1)(n+2)}{2(2n+1)(2n+3)(2n+5)}.
\end{equation}
Consequently,
\begin{equation*}
\frac34\int_{-1}^1(3z^2-1)\Psi^2\,dz = \sum_{n\ge0}\frac34(3\sigma_n-1)h_na_n^2+2\sum_{n\ge0}\frac94\rho_nh_{n+2}a_na_{n+2}.
\end{equation*}

Combining \eqref{app:fourth-diag} and \eqref{app:potential-diag}, the diagonal coefficient in \eqref{Q0form} is given by
\begin{equation*}
\begin{split}
d_n &= \frac{n(n-1)(n+1)(n+2)(n^2+n+3)}{2(2n-1)(2n+1)(2n+3)}+\frac{3n(n+1)}{(2n-1)(2n+1)(2n+3)} \\
& =\frac{n^2(n+1)^2(n^2+n+1)}{2(2n-1)(2n+1)(2n+3)}.
\end{split}
\end{equation*}
Similarly, combining \eqref{app:fourth-cross} and \eqref{app:potential-cross}, the off-diagonal coefficient is
\begin{equation*}
\begin{split}
e_n &=-\frac{(n-1)n(n+1)(n+2)(n+3)(n+4)}{4(2n+1)(2n+3)(2n+5)}+\frac{9(n+1)(n+2)}{2(2n+1)(2n+3)(2n+5)} \\
& =-\frac{(n+1)(n+2)(n^4+6n^3+5n^2-12n-18)}{4(2n+1)(2n+3)(2n+5)}.
\end{split}
\end{equation*}
With these formulas for $d_n$ and $e_n$, and noting that each off-diagonal term appears twice in the sum, we obtain \eqref{Q0form}.


\begin{thebibliography}{10}

\bibitem{AB}
	\newblock	W. Arendt and C. J. K. Batty,
	\newblock	\textit{Tauberian theorems and stability of one-parameter semigroups},
	\newblock	Trans. Amer. Math. Soc.
	\newblock	306 (1988),
	\newblock	pp. 837--852.

\bibitem{BKT}
	\newblock	J. Bricmont, A. Kupiainen, and J. Taskinen,
	\newblock	\textit{Stability of Cahn--Hilliard fronts},
	\newblock	Comm. Pure Appl. Math.
	\newblock	Vol. LII (1999),
	\newblock	pp. 839--871.

\bibitem{CH}
	\newblock	J.~W. Cahn and J.~E. Hilliard,
	\newblock	\textit{Free energy of a nonuniform system. I. Interfacial free energy},
	\newblock	J. Chem. Phys.
	\newblock	28 (1958),
	\newblock	pp. 258--267.

\bibitem{Cahn}
	\newblock	J.~W. Cahn,
	\newblock	\textit{On spinodal decomposition},
	\newblock	Acta Metall.
	\newblock	9 (1961),
	\newblock	pp. 795--801.

\bibitem{CCO}
	\newblock	E. A. Carlen, M. C. Carvalho, and E. Orlandi,
	\newblock	\textit{A simple proof of stability of fronts for the Cahn--Hilliard equation},
	\newblock	Commun. Math. Phys.
	\newblock	224 (2001),
	\newblock	pp. 323--340.

\bibitem{EZ}
	\newblock	C.~M. Elliott and S.~Zheng,
	\newblock	\textit{On the Cahn--Hilliard equation},
	\newblock	Arch. Ration. Mech. Anal.
	\newblock	96 (1986),
	\newblock	pp. 339--357.

\bibitem{EG}
	\newblock	C.~M. Elliott and H.~Garcke,
	\newblock	\textit{On the Cahn--Hilliard equation with degenerate mobility},
	\newblock	SIAM J. Math. Anal.
	\newblock	27 (1996),
	\newblock	pp. 404--423.


\bibitem{Ho}
	\newblock	P. Howard,
	\newblock	\textit{Asymptotic behavior near transition fronts for equations of generalized Cahn--Hilliard form},
	\newblock	Commun. Math. Phys.
	\newblock	269 (2007),
	\newblock	pp. 765--808.

\bibitem{Ho1}
	\newblock	P. Howard,
	\newblock	\textit{Asymptotic behavior near planar transition fronts for the Cahn--Hilliard equation},
	\newblock	Physica D
	\newblock	229 (2007),
	\newblock	pp. 123--165.

\bibitem{Ho2}
	\newblock	P. Howard,
	\newblock	\textit{Spectral analysis of planar transition fronts for the Cahn--Hilliard equation},
	\newblock	J. Differential Equations
	\newblock	245 (2008),
	\newblock	pp. 594--615.

\bibitem{Ho20}
	\newblock	P. Howard,
	\newblock	\textit{Spectral analysis of stationary solutions of the Cahn--Hilliard equation},
	\newblock	Adv. Differential Equations
	\newblock	14 (2009),
	\newblock	pp. 87--120.

\bibitem{Ho3}
	\newblock	P. Howard,
	\newblock	\textit{Spectral analysis for transition front solutions in multidimensional Cahn--Hilliard systems},
	\newblock	J. Differential Equations
	\newblock	257 (2014),
	\newblock	pp. 3448--3465.

\bibitem{Ho4}
	\newblock	P. Howard,
	\newblock	\textit{Short-time existence theory toward stability for nonlinear parabolic systems},
	\newblock	J. Evol. Equ.
	\newblock	15 (2015),
	\newblock	pp. 403--456.

\bibitem{Ho5}
	\newblock	P. Howard,
	\newblock	\textit{Stability of transition front solutions in multidimensional Cahn--Hilliard systems},
	\newblock	J. Nonlinear Sci.
	\newblock	26 (2016),
	\newblock	pp. 619--661.

\bibitem{Ho6}
	\newblock	P. Howard,
	\newblock	\textit{Linear stability for transition front solutions in multidimensional Cahn--Hilliard systems},
	\newblock	J. Dyn. Diff. Equat.
	\newblock	29 (2017),
	\newblock	pp. 895--955.

\bibitem{HK}
	\newblock	P. Howard and B. Kwon,
	\newblock	\textit{Spectral analysis for transition front solutions in Cahn--Hilliard systems},
	\newblock	Discrete Contin. Dyn. Syst.
	\newblock	32 (2012),
	\newblock	pp. 125--166.

\bibitem{HK1}
	\newblock	P. Howard and B. Kwon,
	\newblock	\textit{Asymptotic $L^p$ stability for transition fronts in Cahn--Hilliard systems},
	\newblock	J. Differential Equations
	\newblock	252 (2012),
	\newblock	pp. 5814--5831.

\bibitem{HK2}
	\newblock	P. Howard and B. Kwon,
	\newblock	\textit{Asymptotic stability analysis for transition front solutions in Cahn--Hilliard systems},
	\newblock	Physica D
	\newblock	241 (2012),
	\newblock	pp. 1193--1222.

\bibitem{KV0}
	\newblock	M.-J. Kang and A. Vasseur,
	\newblock	\textit{Criteria on contractions for entropic discontinuities of systems of conservation laws},
	\newblock	Arch. Ration. Mech. Anal.,
	\newblock	222 (2016),
	\newblock	pp. 343--391.

\bibitem{KV}
	\newblock	M.-J. Kang and A. Vasseur,
	\newblock	\textit{$L^2$-contraction for shock waves of scalar viscous conservation laws},
	\newblock	Ann. Inst. H. Poincar\'e Anal. Non Lin\'eaire,
	\newblock	34 (2017),
	\newblock	pp. 139--156.

\bibitem{KV1}
	\newblock	M.-J. Kang and A. Vasseur,
	\newblock	\textit{Contraction property for large perturbations of shocks of the barotropic Navier--Stokes system},
	\newblock	J. Eur. Math. Soc.
	\newblock	23, 2 (2021),
	\newblock	pp. 585--638.

\bibitem{KV2}
	\newblock	M.-J. Kang and A. Vasseur,
	\newblock	\textit{Uniqueness and stability of entropy shocks to the isentropic Euler system in a class of inviscid limits from a large family of Navier--Stokes systems},
	\newblock	Invent. Math.,
	\newblock	224 (2021),
	\newblock	pp. 55--146.


\bibitem{LP}
	\newblock	Yu. I. Lyubich and V\~u Qu\^oc Phong,
	\newblock	\textit{Asymptotic stability of linear differential equations on Banach spaces},
	\newblock	Studia Math.
	\newblock	88 (1988),
	\newblock	pp. 37--42.

\bibitem{NCS}
	\newblock	A.~Novick-Cohen and L.~A. Segel,
	\newblock	\textit{Nonlinear aspects of the Cahn--Hilliard equation},
	\newblock	Physica D
	\newblock	10 (1984),
	\newblock	pp. 277--298.

\bibitem{MZ1}
	\newblock	C. Mascia and K. Zumbrun,
	\newblock	\textit{Pointwise Green function bounds for shock profiles of systems with real viscosity},
	\newblock	Arch. Ration. Mech. Anal.
	\newblock	169 (2003),
	\newblock	pp. 177--263.
	
\bibitem{MZ2}
	\newblock	C. Mascia and K. Zumbrun,
	\newblock	\textit{Stability of large-amplitude viscous shock profiles for hyperbolic-parabolic systems},
	\newblock	Arch. Ration. Mech. Anal.
	\newblock	172 (2004),
	\newblock	pp. 93--131.
	
\bibitem{MZ3}
	\newblock	C. Mascia and K. Zumbrun,
	\newblock	\textit{Stability of small-amplitude shock profiles of symmetric hyperbolic-parabolic systems},
	\newblock	Comm. Pure Appl. Math.
	\newblock	57 (2004),
	\newblock	pp. 841--876.

\bibitem{OSW}
    \newblock F. Otto, S. Scholtes, and M. G. Westdickenberg,
    \newblock \textit{Optimal $L^1$-type relaxation rates for the Cahn--Hilliard equation on the line},
    \newblock SIAM J. Math. Anal.
    \newblock 51 (2019),
    \newblock no. 6, pp. 4645--4682.

\bibitem{OW}
    \newblock F. Otto and M. G. Westdickenberg,
    \newblock \textit{Relaxation to equilibrium in the one-dimensional Cahn--Hilliard equation},
    \newblock SIAM J. Math. Anal.
    \newblock 46 (2014),
    \newblock no. 1, pp. 720--756.



\bibitem{ZH}
	\newblock	K.~Zumbrun and P.~Howard,
	\newblock	\textit{Pointwise semigroup methods and stability of viscous shock waves},
	\newblock	Indiana Univ. Math. J.
	\newblock	47 (1998),
	\newblock	no.~3, pp. 741--871.

\end{thebibliography}
\end{document}